\documentclass[11pt]{amsart}
\usepackage{amsmath,amsfonts,amssymb,amsthm}
\usepackage[utf8]{inputenc} 
\usepackage{color} 
\usepackage{tikz}
\usepackage{mathtools}
\providecommand{\abs}[1]{\lvert#1\rvert}
\newtheorem{lemma}{Lemma}[section]
\newtheorem{proposition}[lemma]{Proposition}
\newtheorem{theorem}[lemma]{Theorem}

\newtheorem{corollary}[lemma]{Corollary}
\newtheorem{definition}[lemma]{Definition}

\newtheorem{remark}[lemma]{Remark}

\usepackage{euscript} 
\let\matheu=\EuScript
\usepackage{mathrsfs}
\usepackage{hyperref}
\usepackage{tikz-cd}
\newcommand{\cl}{\text{cl}}
\newcommand{\FS}{\text{FS}}
\newcommand{\Coh}{\text{Coh}}
\newcommand{\Perf}{\text{Perf}}
\newcommand{\dg}{\text{dg}}
\newcommand{\D}{\text{D}}
\newcommand{\id}{\text{id}}
\hypersetup{colorlinks=true,linkcolor=blue,citecolor=purple}
\tikzset{every picture/.style={line width=0.75pt}}
\title{SYZ for index $1$ Fano hypersurfaces in projective space}
\author{Mohamed El Alami}
\begin{document}
\maketitle
\begin{abstract}
    We study homological mirror symmetry of the singular hypersurface $X_0=V(t^{n+1}-x_0\dotsi x_n)\subseteq\mathbb{P}^{n+1}$. Following an SYZ type approach, we produce an LG-model, whose Fukaya-Seidel category recovers line bundles on $X_0$. As a byproduct of our approach, we answer a conjecture of N. Sheridan about generating the small component of the Fukaya category of the \emph{smooth} index 1 Fano hypersurface in $\mathbb{P}^{n+1}$, without bounding co-chains.
\end{abstract}
\tableofcontents
\newpage
\section{Introduction}
In his seminal work \cite{sheridan-fano-mirror-symmetry}, N. Sheridan studied homological mirror symmetry for all Fano hypersurfaces $X_d$ of degree $d$ in projective space $\mathbb
{P}^{n+1}$, where $1\leq d \leq n+1$. When $d$ is fixed, all such hypersurfaces are symplectomorphic and that makes the $A$-side. The $B$-side in his work is a Landau-Ginzburg model $(Y_d,W_d)$, and the main theorem of \cite{sheridan-fano-mirror-symmetry} is an exact equivalence of triangulated categories:
\begin{equation}\label{sheridan-HMS}
    \text{D}^{\pi}\text{Fuk}(X_d) \cong \text{D}^b_{\text{sing}}(Y^{}_d,W^{}_d).
\end{equation}

The key component of (\ref{sheridan-HMS}) is a chain of Lagrangian spheres in $X_d$ that N. Sheridan constructs building upon his earlier work in \cite{sheridan-pair-of-pants,Sheridan-cy-mirror-symmetry}. As these Lagrangians are geometrically rigid, he resorts to studying their algebraic deformations using weak bounding co-chains in order to compute the mirror LG-model $(Y_d,W_d)$.

In the present work, we mostly investigate the other direction of mirror symmetry, i.e when $X_d$ is the $B$-side. In doing so, we explore a more direct approach following the lines of Strominger-Yau-Zaslow. We limit our attention to the index $1$ Fano case, i.e. $d=n+1$. 

There is a simple construction of a \emph{partial} SYZ-fibration on $X_{n+1}\subseteq \mathbb{P}^{n+1}$, which is obtained by projecting away from a point to a hyperplane $\mathbb{P}^n\subseteq \mathbb{P}^{n+1}$. When the branch locus is sufficiently close to the toric boundary of $\mathbb{P}^n$, one can lift some of the Clifford tori $L_\cl\subseteq \mathbb{P}^n$ to Lagrangian tori $L\subseteq X_{n+1}$.
Though it is partial, this fibration can be made arbitrarily large by pushing the branch locus of the projection closer to the toric boundary of $\mathbb{P}^n$. At its limit, this process degenerates $X_{n+1}$ to a singular toric hypersurface $X_0$, which has the following defining equation in homogeneous coordinates:
\begin{equation*}
    X_0 = V(t^{n+1}-x_0\dotsi x_n)\subseteq \mathbb{P}^{n+1}.
\end{equation*}

Our first main result is a computation of the super-potential $W:(\mathbb{C}^*)^n\rightarrow \mathbb{C}$ associated with this partial SYZ-fibration. The Laurent polynomial $W$ packages all the counts of holomorphic discs in $X_{n+1}$, bounded by Lagrangian fibers, and whose of Maslov index is $2$. 

\begin{theorem}
There is a partial SYZ-fibration on $X_{n+1}$ whose associated super-potential has the formula:
\begin{equation*}
    W = \frac{(1+y_1+\dots+y_n)^{n+1}}{y_1\dotsi y_n} - (n+1)!.
\end{equation*}
\end{theorem}

We note that, up to the $-(n+1)!$ translation term, this result agrees with the expected Hori-Vafa mirror for the toric hypersurface $X_0$.

Our counts of Maslov index $2$ discs shed some light on a question regarding the HMS equivalence in (\ref{sheridan-HMS}). To put it in context, recall that $\text{D}^{\pi}\text{Fuk}(X_{n+1})$ splits into components corresponding with the eigenvalues of quantum multiplication by $c_1(TX_{n+1})$. There are two such eigenvalues:
A small one $w_s$ which is a non-degenerate singularity in the mirror, and a big one $w_b$ which is a more complicated singularity. The statement in (\ref{sheridan-HMS}) is therefore made of an equivalence over the small eigenvalue (also called the small component), and another one over the big eigenvalue (similarly called the big component). Sheridan's Lagrangian spheres naturally see the big component, and there they generate. However, in order to get them to see the small component, they require algebraic deformations using weak bounding co-chains. At the end of his paper \cite{sheridan-fano-mirror-symmetry}, conjecture B.2, the author contemplates the possibility of covering the small component using honest monotone Lagrangians without bounding co-chains. It turns out that the \emph{partial} SYZ fibration we produce has a central monotone fiber, and we use it to show the following result.

\begin{theorem}\label{HMS2}
The smooth index $1$ Fano hypersurface $X_{n+1}\subseteq\mathbb{P}^{n+1}$ contains a monotone Lagrangian torus that split-generates the small component of its Fukaya category.
\end{theorem}

Note that the case $n=2$ of the theorem above has been established in the work of D. Tonkonog and J. Pascaleff on Lagrangian mutations, see \cite{Pascaleff-Tonkonog}.

Next, we view the super-potential $W$ as the $A$-side, and we study a homological mirror symmetry correspondence between the singular toric limit $X_0$, and the Landau-Ginzburg model $((\mathbb{C}^*)^{n},W)$. Our results in this direction can be summarized as follows.

\begin{theorem}\label{HMS 1}
There is a collection of Lefschetz thimbles $L_i$ in $((\mathbb{C}^*)^n,W)$ such that:
\begin{equation*}
    HW(L_i,L_j) \simeq \hom_{X_0}(\matheu{O}_{X_0}(i),\matheu{O}_{X_0}(j)).
\end{equation*}
Furthermore, the isomorphisms above are compatible with the relevant product structures. 
\end{theorem}

Our approach relies on understanding how the branched covering map $\phi:X_0\rightarrow \mathbb{P}^n$ corresponds (under mirror symmetry) to an \emph{unbranched} quotient map:
\begin{equation*}
\pi: ((\mathbb{C}^*)^n,W_\cl)\rightarrow ((\mathbb{C}^*)^n,W),    
\end{equation*}
where $W_\cl$ is the super-potential associated with a Clifford torus in projective space $L_\cl\subseteq \mathbb{P}^n$.
\vspace{.5cm}\\
\noindent \textbf{Outline of the paper.} In section 2, we recall some facts about Maslov classes, their behavior with respect to anti-canonical divisors and branched coverings. We use these ideas to construct a monotone Lagrangian torus $L$ in a \emph{nearby} smoothing of the singular hypersurface $X_0$. In section 3, we compute the super-potential $W$ associated with $L$. This computation has two parts. First, we make an educated \emph{guess} of the correct count $m_{0,\beta}(L)$, by mapping the relevant Maslov index $2$ discs down to projective space $\mathbb{P}^n$, using the cyclic covering map $\phi: X_0\rightarrow \mathbb{P}^n$. Then, we explain a transversality argument that confirms that our guesses are indeed actual counts of Fredholm regular curves. In section 4, we view the smooth index $1$ Fano hypersurface as the A-side, and we compute Fukaya's $A_{\infty}$-algebra associated with the monotone Lagrangian torus $L$. We show in particular that $L$ split-generates the small component. In section 5, the super-potential $W$ is placed on the A-side. We compute the (partially) wrapped Floer cohomology of Lagrangian thimbles in the Fukaya-Seidel category associated with $W$, and we explain how they correspond with line bundles on $X_0$. \\

\noindent \textbf{Acknowledgement.} The author would like to express his gratitude to Mark McLean for his guidance and influence, and especially for suggesting the work of Cieleback and Mohnke in \cite{CM}, which is the original inspiration for this project.


\section{Construction of Lagrangian tori}

\subsection{Topological preliminaries}
\subsubsection{Intersection numbers}
We begin by recalling, and setting notation for intersection numbers as this will be used extensively throughout this section. Let $X$ be a smooth oriented compact manifold, and let $Y\subseteq X$ be a codimension $2$ submanifold. We always think of $Y$ as the zero set of a smooth section $s\in \Gamma(X,\mathscr{L})$ of a complex line bundle $\mathscr{L}\rightarrow X$. Let $u:\Sigma \rightarrow X$ be a smooth map from a compact Riemann surface $\Sigma$, such that:
\begin{equation}\label{transversality1}
   u(\partial \Sigma)\cap Y = \emptyset.
\end{equation}
The intersection number $u\cdot Y$ is defined to be the signed counted of zeroes of the restriction $u^*s$ of the section $s$ to $\Sigma$. This may require a small perturbation of $u$ to ensure that the pullback $u^*s$ is transverse to the zero section of $u^*\mathscr{L}\rightarrow \Sigma$. This intersection number does not change under homotopies of $u$ that preserve the boundary condition (\ref{transversality1}). When the Riemann surface $\Sigma$ has no boundary, the intersection number has the following integral formula:
\begin{equation*}
    u\cdot Y = \langle c_1(\mathscr{L}),u\rangle .
\end{equation*}
As an example of how these intersection numbers work, we present a quick proof of the Riemann-Hurwitz theorem.
Let $\phi:X\rightarrow Y$ be a finite map between smooth projective varieties with ramification locus $R$. Let $u:\Sigma \rightarrow X$ be a holomorphic map from a closed Riemann surface $\Sigma$. Then:
\begin{equation*}
    c^X_1(u) = c^Y_1(\phi\circ u) - u\cdot R,
\end{equation*}
where $c^X_1=c_1(TX)$ is the first Chern class of the tangent bundle, and:
\begin{equation*}
    c_1^X(u) = \int_\Sigma u^*c^X_1.
\end{equation*}

Indeed, the ramification locus is the zero set of the section $\wedge^n d\phi$ of the line bundle:
\begin{equation*}
   \mathscr{L} = \wedge^n TX \otimes (\wedge^n \phi^*TY)^{-1}.
\end{equation*}
Therefore:
\begin{equation*}
u\cdot R = \langle c_1(\wedge^n TX \otimes (\wedge^n \phi^*TY)^{-1}) , u \rangle,  
\end{equation*}
and the classical Riemann-Hurwitz formula follows.

This formula has a relative analogue as well: Let $L\subseteq X$ and $K\subseteq Y$ be totally real sub-manifolds such that $R\cap L=\emptyset$, and $\phi(L)\subseteq K$. Let $u:\Sigma\rightarrow Y$ be a map from a Riemann surface with boundary $\Sigma$ such that $u(\partial \Sigma)\subseteq L$. Then:
\begin{equation*}
    \mu_L^X(u) = \mu_K^Y(\phi\circ u) - 2u\cdot R,
\end{equation*}
where $\mu$ is the Maslov class, which we will recall soon. The proof is identical.

The next lemma will be used implicitly in our calculations. The proof is a direct application of (and in fact the reason we recalled) the definition of intersection numbers.
\begin{lemma}
Let $\phi:X\rightarrow Y$ be a finite map of smooth projective varieties with branch locus $H$, and let $D_X=\phi^{-1}(H)$ be its (possibly non-reduced) pre-image. Then for any disc map $u:(D,\partial D)\rightarrow X$ with $u(\partial D)\cap D_X = \emptyset$, we have:
\begin{equation*}
    u\cdot D_X = (\phi\circ u)\cdot H.
\end{equation*}\qed
\end{lemma}
\subsubsection{Maslov numbers}
Let $(X,\omega)$ be a K\"ahler manifold. The \emph{primary} Maslov class associated with an oriented totally real subspace $L\subseteq X$, is a $\mathbb{Z}$-module homomorphism: 
\begin{equation*}
    \mu_L^X : H_2(X,L)\rightarrow \mathbb{Z}.
\end{equation*}
For any map $u:(\Sigma,\partial \Sigma)\rightarrow (X,L)$, it is defined as a relative Euler characteristic:
\begin{equation*}
    \mu^X_L(u)=\chi ((\wedge_{\mathbb{C}}^n u^*TX)^{\otimes 2},(\wedge_{\mathbb{R}}^n u^*TL)^{\otimes 2}).
\end{equation*}
This means counting zeros of a generic section of the complex line bundle $(\wedge_{\mathbb{C}}^n u^*TX)^{\otimes 2}$ over $\Sigma$, whose restriction to $\partial \Sigma$ belongs to the real sub-bundle $(\wedge_{\mathbb{R}}^n u^*TL)^{\otimes 2}$. In particular, when $u$ is the class of a closed Riemann surface, the Maslov number is twice the Chern number:
\begin{equation*}
    \mu^X_L(u) = 2\langle c_1(X),[u] \rangle.
\end{equation*}

In our context, it will be equally important to consider a \emph{secondary} Maslov class:
\begin{equation}\label{secondary-maslov-class}
    \eta^{X,H}_L : H_1(L,\mathbb{Z}) \rightarrow \mathbb{Q}.
\end{equation}
This one is more relevant in the complement of a hypersurface $H\subseteq X$, that is a multiple of an anti-canonical divisor. In other words:
\begin{equation*}
    \matheu{O}(H)=K_X^{-N},
\end{equation*}
for some positive integer $N$. To construct it, chose a smooth trivialization $s$ of $K_{X\backslash H}^{-N}$ and an orientation $n$-form $\alpha$ for $L$ (The orientation assumption is not necessary but it simplifies the discussion a bit). We can compare the two trivializations using the embedding $L\hookrightarrow X$ and in doing so, we obtain an argument function:
\begin{align*}
    \arg_L: L &\rightarrow \mathbb{C}^*\\
            x &\mapsto \alpha^{\otimes N}/s.
\end{align*}
Set $A=X\backslash H$, then the \emph{secondary} Maslov class of the pair $(A,L)$, viewed as a cohomology element $\eta^A_L\in H^1(L,\mathbb{Q})$ is:
\begin{equation*}
    \eta_L^{A} = \frac{2}{N}\arg_L^*(d\theta).
\end{equation*}
Note that the compactification $X$ of $A$ plays no role in the construction so far; all we needed is an affine variety whose $c_1(A)$ is torsion. Assuming $L$ is connected, the class we have constructed only depends on the choice of the trivialization $s$, which is sometimes called a \emph{grading} for $A$ (see \cite{PL-theory}, for instance).

\begin{lemma}\label{Maslov-from-cover-to-base}
Let $\phi:A\rightarrow B$ be an unbranched covering map of smooth affine varieties with $Nc_1(B)=0$, and let $L\subseteq A$ and $K\subseteq B$ be totally real sub-manifolds such that $\phi(L)\subseteq K$. Then we have:
\begin{equation*}
    \phi^*(\eta^B_K)=\eta^A_L,
\end{equation*}
for appropriately chosen trivializations.
\end{lemma}
\begin{proof}
 Any choice of a trivialization of $K_B^{-N}$ can be pulled-back to a choice of trivialization for $K_A^{-N}$, and the same goes for orientations of $K$. With such choices, we ensure that $\arg_L = \phi^*\arg_K$, and the lemma follows.
\end{proof}

In the presence of a compactification $(X,H)$ of the variety $A=X\backslash H$, such that $H_1(X)=0$, it is possible to arrange for $\eta$ to be choice-independent. Instead of a trivialization of $K_A^{-N}$, one instead chooses a smooth section $s$ of $K_X^{-N}$ that is nowhere vanishing on $A$, and the previous construction results in the desired, choice-independent, Maslov class (\ref{secondary-maslov-class}). This is made evident by the next result.

\begin{lemma}\label{maslov-number-formula}
Let $X$ be a smooth projective variety, $H\subseteq X$ a hypersurface, and $L\subseteq X$ an oriented totally real submanifold such that $L\cap H = \emptyset$. Furthermore, assume that there exists a natural number $N$ such that:
\begin{equation*}
    \matheu{O}(H)=K_X^{-N}.
\end{equation*}
Then for any disc $u:(D,\partial D) \rightarrow (X,L)$, we have:
\begin{equation*}
    \mu_L^X(u) = \eta_L^{X,H}(\partial u) + \frac{2}{N} u\cdot H.
\end{equation*}
\end{lemma}
\begin{proof}
We start by fixing an orientation form $\alpha$ for $L$. Let $s$ be a smooth section of $K_X^{-N}$ vanishing along $H$. Then, the secondary Maslov number of a disc $u:(D,\partial D)\rightarrow (X,L)$ is:
\begin{equation*}
    \eta_L^{X,H}(\partial u) =\frac{2}{N}\deg(\arg_L\circ\ \partial u : \partial D\rightarrow \mathbb{C}^*).
\end{equation*}
Let $\arg^u_L:D\rightarrow \mathbb{C}$ be an extension of $\arg_L\circ\ \partial u$. Then $\arg^u_L\cdot s$ is a relative section of the bundle pair:
$$
((\wedge_{\mathbb{C}}^n u^*TX)^{\otimes N},(\wedge_{\mathbb{R}}^n u^*TL)^{\otimes N}).
$$
It follows that:
\begin{align*}
    \mu^X_L(u) 
    &= \frac{2}{N}\chi ((\wedge_{\mathbb{C}}^n u^*TX)^{\otimes N},(\wedge_{\mathbb{R}}^n u^*TL)^{\otimes N}) \\
    &=\frac{2}{N}\# (\arg^u_L\cdot s)^{-1}(0)\\
    &= \frac{2}{N}\left(\deg(\arg_L\circ\ \partial u : \partial D\rightarrow \mathbb{C}^*) + u\cdot H \right).
\end{align*}
The Maslov number formula then follows.
\end{proof}

\begin{remark}
Most of this section's content has previously appeared in the literature. For example:
\begin{itemize}
    \item[-] When defining the secondary Maslov class, the choice of trivialization of (a multiple of) the canonical bundle is called a \emph{grading}, and the construction we made appears for example in P.Seidel's book \cite{PL-theory}. 
    \item[-] The Maslov number formula we produced also has analogues in the literature pertaining to mirror symmetry in log Calabi-Yau varieties, it appears for instance in D. Auroux's paper \cite{auroux-t-duality}.
\end{itemize}
\end{remark}
\subsection{Monotone Lagrangian tori in branched covers}
\subsubsection{Maslov numbers and branched covers}
Let $X$ be an $n$-dimensional smooth projective variety. It is standard that $X$ admits a finite map to projective space of the same dimension. Such a map is obtained for example by composing an embedding to $\mathbb{P}^N$, with a generic linear projection from a codimension $(n+1)$-plane. Such finite maps will (almost) always be branched, and a preliminary study of the branch locus is necessary for our purposes. We restrict ourselves to the following context:

\begin{lemma}\label{Branch-covering-setup}
Let $X$ be a smooth Fano variety and $m>0$ a positive integer such that $\abs{mK_X^{-1}}$ is very ample. Suppose we use elements of this linear system to produce a finite map:
\begin{equation*}
  \phi :  X \rightarrow \mathbb{P}^n.
\end{equation*}
Then, the branch locus $B$ of $\phi$ is a (possibly non-reduced) hypersurface of degree:
\begin{equation*}
    \deg(B) = \left(n+1-\frac{1}{m}\right)\deg(X).
\end{equation*}
\end{lemma}

\begin{proof}
Since $mK_X^{-1}$ is very ample, it can be used to produce an embedding $X\subseteq \mathbb{P}^N$ for some large $N$, and the finite map $\phi$ is the composition of this embedding with a linear projection. Since $B$ is the image of the ramification locus $R$ under a linear projection, we must have:
\begin{equation}\label{degree-comp}
    \deg(B\subseteq \mathbb{P}^n) = \deg(R\subseteq \mathbb{P}^N).
\end{equation}
Using the Riemann-Hurwitz formula, we have:
\begin{equation*}
    \text{pd}_X(R) = c_1(K_X\otimes \phi^*K_{\mathbb{P}^n}^{-1}),
\end{equation*}
where $\text{pd}_X$ is Poincarr\'e duality on $X$. It follows that:
\begin{align*}
    mR= (m(n+1)-1)\text{pd}_X(c_1(\matheu{O}_X(1))).
\end{align*}
But we know that $\text{pd}_X(\matheu{O}_X(1))$ is a hyperplane section of $X$, and therefore:
\begin{equation*}
    m\deg(R) = (m(n+1)-1)\deg(X).
\end{equation*}
The lemma now follows from the observation in (\ref{degree-comp}). 
\end{proof}

\begin{remark}
The degree formula above should be known in the literature but we could not find a reference for it. A famous instance of the formula is the case of a general projection of cubic surface branching over a sextic curve.
\end{remark}
We will only note the following consequence of the previous degree formula.
\begin{corollary}
Let $X$ be a Fano variety of dimension $n$. Then $X$ admits a finite branched covering map whose branch locus $B$ is of degree divisible by $n+1$. \qed
\end{corollary} 

Returning back to our branched covering morphism, let us denote by $D_X=\phi^{-1}(B)$ the (possibly non-reduced) extended ramification locus. Let $L\subseteq \mathbb{P}^n$ be a totally real torus that is disjoint from the branch locus, and let $L_X$ be (a component of) its pre-image. Then $L_X\subseteq X$ is itself a totally real torus. We would like to relate Maslov numbers of the pair $(X,L_X)$ to those of $(\mathbb{P}^n,L)$.
\begin{lemma}
Let $u:(D,\partial D)\rightarrow (X,L_X)$ be a disc map, and let $v=\phi\circ u : (D, \partial D) \rightarrow (\mathbb{P}^n,L)$ be its image in projective space. Then:
\begin{equation}\label{Maslov number formula}
    \mu^X_{L_X}(u) = \mu_L^{\mathbb{P}^n}(v)-\frac{2}{\deg(B)}\left(n+1-\frac{1}{m}\right) v\cdot B.
\end{equation}
\end{lemma}

\begin{proof}
This is direct computation using the results of Lemma \ref{Maslov-from-cover-to-base} and Lemma \ref{maslov-number-formula} (applied both to $X$ and to $\mathbb{P}^n$):
\begin{align*}
    \mu^X_{L_X}(u) &= \eta_{L_X}^{X,D_X}(\partial u) + \frac{2}{m\deg(B)}u \cdot D_X\\
                   &= \eta_{L}^{\mathbb{P}^n,B}(\partial v) + \frac{2}{m\deg(B)}v\cdot B\\
                   &= \mu_L^{\mathbb{P}^n}(v)-\frac{2(n+1)}{\deg(B)}v\cdot B + \frac{2}{m\deg(B)}v\cdot B.
\end{align*}
Rearranging some of the terms results in the desired identity.
\end{proof}
\subsubsection{Weakly monotone tori}
So far, our discussion does not involve the K\"ahler structure. We keep it that way by introducing the notion of \emph{weakly monotone} totally real sub-manifolds.
\begin{definition}
Given a pair $(X,D_X)$ of a smooth projective variety together with a hypersurface, we say that a totally real sub-manifold $L_X\subseteq X\backslash D_X$ is weakly monotone, if there is a rational number $\lambda \in \mathbb{Q}$ such that for any disc $u:(D,\partial D) \rightarrow (X,L_X)$, one has:
\begin{equation*}
    \mu_{L_X}(u) = 2\lambda u\cdot D_X.
\end{equation*}
\end{definition}
\begin{remark}
This is like saying that $L_X$ is monotone with respect to a K\"ahler form that is a Dirac-Delta along $D_X$.
\end{remark}
We note that this definition only makes sense when $D_X$ is (numerically) a multiple of the anti-canonical class, and in that case, the constant $\lambda$ must be the inverse of said multiple.

The easiest way of obtaining weakly-monotone Lagrangians comes from toric geometry. We take the example of $\mathbb{P}^n$ with homogeneous coordinates $[z_0:z_1:\dots:z_n]$, which is the most relevant one to our construction. It admits a toric structure with toric boundary equal to a union of $n+1$ hyperplanes:
\begin{equation*}
    H = \bigcup_{i=0}^n \{ z_i=0\}.
\end{equation*}
The toric fibers are parametrized by vectors $\boldsymbol{r}\in \mathbb{R}_{>0}^{n+1}$, and they take the form:
\begin{equation*}
    L_{\boldsymbol{r}}=\{[z_0:\dots:z_n] \ | \ r_0^{-1}\abs{z_0}=\dots=r_n^{-1}\abs{z_n} \}.
\end{equation*}
All of these tori are totally real. To see that they are weakly monotone in $(\mathbb{P}^n,H)$, we can use a generating set of the relative homology group $H_2(\mathbb{P}^n,L_{\boldsymbol{r}})$, such as the collection of holomorphic discs given by:
\begin{align}\label{generators of relative H2}
    u_k(\boldsymbol{r}) : (D,\partial D) &\rightarrow (\mathbb{P}^n,L_{\boldsymbol{r}})\\
    z &\mapsto [r_0:\dots:r_{k}z:\dots:r_{n}].\notag
\end{align}
Note that these classes add up to the spherical class that generates $H_2(\mathbb{P}^n)$. They each have Maslov number $2$, and they each intersect $H$ exactly once. It follows that for all discs $[u]\in H_2(\mathbb{P}^n,L_{\boldsymbol{r}})$:
\begin{equation*}
    \mu_{L_{\boldsymbol{r}}}(u)=2u\cdot H,
\end{equation*}
and thus $L_{\boldsymbol{r}}\subseteq (\mathbb{P}^n,H)$ is weakly monotone. The torus $L_\cl$ corresponding to $\boldsymbol{r}=(1,1,\dots,1)$ is usually called the Clifford torus, and it is the only one among these tori that is monotone with respect to the Fubini-Study metric.

The existence of a weakly monotone torus  $L_{\boldsymbol{r}} \subseteq (\mathbb{P}^n,B)$ has an obstruction coming from the degree of the hypersurface $B$.
\begin{lemma}
Let $B\subseteq \mathbb{P}^n$ be a hypersurface. Assume there exists a vector $\boldsymbol{r}$ such that $L_{\boldsymbol{r}} \subseteq (\mathbb{P}^n,B)$ is a weakly monotone totally real torus. Then $\deg(B)$ is divisible by $n+1$.
\end{lemma}

\begin{proof}
This follows immediately from the disc classes $[u_k(\boldsymbol{r})]\in H_2(\mathbb{P}^n,L_{\boldsymbol{r}})$ adding up to a boundary-free class representing a line in $\mathbb{P}^n$. Hence:
\begin{equation*}
    \deg(B) = \sum_{i=0}^n u_k(\boldsymbol{r})\cdot H = (n+1)u_0(\boldsymbol{r})\cdot B.
\end{equation*}
Therefore $\deg(B)$ is divisible by $n+1$.
\end{proof}

The previous construction of weakly monotone tori extends to other hypersurfaces $B\subseteq\mathbb{P}^n$ that are close to (a multiple of) $H$.
\begin{definition}
We call a hypersurface $B\subseteq \mathbb{P}^n$ nearly degenerate if it is a 'small' perturbation of a multiple of $H$.
\end{definition}
We can actually quantify how small the perturbation needs to be. Let $f_0=z_0\dotsi z_n$ be the defining equation of $H$. A small perturbation of $kH$ is a hypersurface $B_f = V(f)$, whose defining is:
\begin{equation*}
    f = f_0^k + h,
\end{equation*}
where $h$ is a homogeneous polynomial of degree $d=nk$, satisfying the inequality:
\begin{equation}\label{nearly-degenerate}
    \abs{h(z_0:\dots:z_n)} < \frac{\abs{z_0}^d+\dots+\abs{z_n}^d}{n}.
\end{equation}

\begin{lemma}
Suppose that $B_f\subseteq \mathbb{P}^n$ is a nearly degenerate hypersurface. Then, the Clifford torus $L_{\cl}$ is disjoint from $B_f$ and is weakly monotone in $(\mathbb{P}^n,B_f)$.
\end{lemma}

\begin{proof}
Indeed, if $[z_0:\dots:z_n]$ is an intersection point of $B_f$ and $L_\cl$, then:
\begin{align*}
    \abs{z_0}^d &= \abs{z_0\dotsi z_n}^k\\
                &= \abs{h(z_0:\dotsi:z_n)}
                < \abs{z_0}^d
\end{align*}
which is a contradiction. Therefore, $L_\cl\cap B_f=\emptyset$. Next, we compute the intersection numbers $u_k\cdot B_f$, by counting (with multiplicity) the zeros of $f\circ u_k:D\rightarrow \mathbb{C}$. Note that for any $z\in D$:
\begin{align*}
  \abs{f\circ u_k(z) - f_0\circ u_k(z)} &=\abs{h\circ u_k(z)}\\
  &<\frac{\abs{z}^d+n-1}{n}.
\end{align*}
In particular, when $z\in\partial D$, we get:
\begin{equation*}
 \abs{f\circ u_k(z) - f_0\circ u_k(z)}<\abs{f_0\circ u_k(z)}.  
\end{equation*}
It follows (by Rouch\'e's theorem) that $f\circ u_k(z)$ and $f_0\circ u_k(z)$ have the same number of zeros and therefore:
\begin{equation*}
    \mu(u_k) = 2u_k\cdot B_f.
\end{equation*}
Since the $(u_k)_k$ generate the relative homology group $H_2(\mathbb{P}^n,L_\cl)$, the statement of the lemma follows.
\end{proof}

Suppose now that we have a finite map $\phi:X\rightarrow \mathbb{P}^n$ as in the setup of Lemma \ref{Branch-covering-setup}, whose branch locus $B$ is nearly degenerate. By the previous lemma, $B$ is disjoint from $L_\cl$, so let $L_X$ be (a connected component of) its pre-image $\phi^{-1}(L_\cl)$. Then we have the following:

\begin{lemma}
The totally real torus $L_X\subseteq (X,R)$ is weakly monotone, where $R=\phi^{-1}(B)$ is the (extended) ramification locus.
\end{lemma}

\begin{proof}
The previous lemma asserts that $L_\cl \subseteq (\mathbb{P}^n,B)$ is weakly monotone. With that in mind, we can use the Maslov number formula (\ref{Maslov number formula}) to see that for any disc $u:(D,\partial D) \rightarrow (X,R)$, one has:
\begin{align*}
    \mu_L(u) 
    &= \frac{2(n+1)}{\deg(B)}v\cdot B - \frac{2}{\deg(B)}\left(n+1-\frac{1}{m}\right) v\cdot B\\
    &= \frac{2}{m\deg(B)}u.R,
\end{align*}
where $v=\phi\circ u$. It follows that $L_X\subseteq (X,R)$ is weakly monotone.
\end{proof}
\subsubsection{Partial Lagrangian fibration}
We now explain how to construct suitable K\"ahler structres on branched covers, so as to make the weakly monotone Lagrangians in our previous discussion into genuine monotone Lagrangians.
\begin{lemma}\label{Kahler-metric}
Let $(Y,\omega)$ be a K\"ahler variety with $[\omega]\in H^2(Y,\mathbb{Z})$, and let $\phi: X\rightarrow Y$ be a finite branched cover. Then, for any neighborhood $U$ of the ramification locus, there exists a K\"ahler form $\omega_X$ on $X$, and a real valued function $\rho:X\rightarrow \mathbb{R}$ with support in $U$, such that:
\begin{equation*}
    \omega_X = \phi^*\omega + dd^c\rho.
\end{equation*}
\end{lemma}

\begin{proof}
Indeed $[\omega]=c_1(\matheu{L})$ is the curvature of some ample line bundle $\matheu{L}\rightarrow Y$ with respect to some Hermitian metric. Since $\phi$ is a finite morphism, the pullback $\phi^*\matheu{L}\rightarrow X$ is necessarily ample, and therefore it admits a positively curved Hermitian metric of its own, and its curvature 2-form $\omega_X$ can be computed as:
\begin{equation*}
    \omega_X=\phi^*\omega + dd^c\psi,
\end{equation*}
where $\psi$ is the (multiplicative) difference between the new positively curved metric and the metric we pull-back from $\matheu{L}$. Now, we choose an open subset $U_1$ of $X$ between the ramification set $R$ and the open neighborhood $U$, such that:
\begin{equation*}
    R\subset U_1 \Subset U.
\end{equation*}
Then, we choose a smooth function $f:X\rightarrow \mathbb{R}$ such that $f=1$ on $U_1$ and $f=0$ outside of $U$.
We now claim that there is a constant $C$ such that:
\begin{equation*}
   \omega_{X,C} = \phi^*\omega + \frac{1}{C}dd^c(f\psi) > 0.
\end{equation*}

Indeed, as long as $C>1$, the Hermitian $2$-form $\omega_{X,C}$ is positive except possibly on $U\backslash U_1$: This is clear outside of $U$, and inside of $U_1$ it can be seen by rewriting:
\begin{equation*}
    \omega_{X,C} = \left(1-\frac{1}{C}\right)\phi^*\omega + \frac{1}{C}\left(\phi^*\omega + dd^c\psi\right).
\end{equation*}
Furthermore, in the limit $C\rightarrow \infty$, the K\"ahler form $\omega_{X,C}$ is also positive on the compact region $\overline{U\backslash U_1}$.
\end{proof}

\begin{remark}
The are versions of this proposition that appear in the literature, e.g. \cite{Sheridan-cy-mirror-symmetry}, or \cite{Auroux-branch-covers-of-cp2}. The advantage we have in our case is that we don't need to worry about the types of singularities of the branch locus.
\end{remark}

The symplectic form constructed in the previous lemma has one key property. Let $L\subseteq X$ be a Lagrangian that is disjoint from the ramification locus. Then for any disc $u:(D,\partial D)\rightarrow (X,L)$, we have an area formula:
\begin{equation*}
    \text{Area}_{\omega_X}(u) = \text{Area}_{\omega}(\phi\circ u).
\end{equation*}

We now specialize all of our previous discussion to the case of the index $1$ Fano hypersurface in projective space.

\begin{proposition}\label{partial-fibration}
Let $(X,\omega_X)\subseteq \mathbb{P}^{n+1}$ be a smooth hypersurface of degree $n+1$, viewed as a monotone symplectic manifold. Then there is a family of anti-canonical symplectic divisor $D_i\subseteq X$, and neighborhoods $D_i\subseteq U_i$ with the following properties:
\begin{itemize}
    \item The open neighborhoods are shrinking, i.e. $\text{Vol}_{\omega_X}(U_i)\rightarrow 0.$
    \item The complement $X\backslash U_i$ admits a Lagrangian torus fiber with a monotone central fiber.
\end{itemize}
\end{proposition}

\begin{proof}
We use homogeneous coordinates $x_0,\dotsi,x_{n},t$ in projective space $\mathbb{P}^{n+1}$, and we set $f_0=x_0\dotsi x_{n}$, and $H=\{t=0\}=\mathbb{P}^n$. Note that $V(f_0)\cap H$ is the usual toric boundary of the $n$-dimensional projective space $H$. We construct a nested sequence of open sets $V_{i+1}\subseteq V_i$, as the pre-images via the Logarithm map, of a shrinking sequence of open neighborhoods of the toric boundary. Then, the open sets $V_i\subseteq H$ shrink to $V(f_0)\cap H$, and the complements $H\backslash V_i$ are all fibered by Lagrangian tori for the Fubini-Study metric $\omega_H$. Next, we construct a sequence $f_i$ of regular homogeneous polynomials in $x_0,\dots,x_{n}$ of degree $n+1$ converging to $f_0$, fast enough to ensure that $V(f_i)\cap H \subseteq V_i$. Define $X_i=V(t^{n+1}-f_i)$, and set $D_i=X_i\cap H$. Forgetting the $t$ variable produces a branched covering $\phi_i: X_i\rightarrow H$ ramified along $D_i$, and the later is contained in the open set $U_i = \phi_i^{-1}(V_i)$. We then apply the construction of K\"ahler metrics in Lemma \ref{Kahler-metric} to produce a symplectic form $\omega_i$ with the property:
\begin{equation*}
    \omega_i = \phi_i^*\omega_H + d\alpha_i,
\end{equation*}
such that $\alpha$ is compactly supported in $U_i$. In particular,
\begin{equation*}
\text{Vol}_{\omega_i}(U_i)=(n+1)\text{Vol}_{\omega_H}(V_i)    
\end{equation*}
converges to $0$. Moreover, $X_i\backslash U_i$ is an unbranched covering of $H\backslash V_i$, and as such, it inherits a Lagrangian torus fibration (see Remark \ref{connected-fibers} below). As a consequence of the Maslov number formula (\ref{Maslov number formula}), and of our choice of symplectic form, the lift of the monotone Clifford torus in $(H,\omega_H)$ will then be monotone. To complete the proof of the proposition, simply use a Moser argument to trivialize the family $(X_i,\omega_i)$.
\end{proof}

\begin{remark}\label{connected-fibers}
We would like to clarify a little more on why the pre-image of a Clifford torus in an index $1$ Fano hypersurface $X_f\subseteq \mathbb{P}^{n+1}$ is connected: In fact, away from the branch locus, the map $\phi$ restricts to an unbranched cyclic covering $\hat{\phi}:X_f\backslash R \rightarrow \mathbb{P}^n\backslash B$  of degree $n+1$. But now $B\subseteq \mathbb{P}^n$ is a smooth hypersurface of degree $n+1$ as well, and it is a classical result (an application of Seifert-Van-Kampen's theorem and Poincarr\'e duality) that:
\begin{equation*}
    \pi_1(\mathbb{P}^n\backslash B) =\mathbb{Z}_{n+1}.
\end{equation*}
Therefore, the map $\hat{\phi}$ above is a universal covering map, and in particular $L_\cl\subseteq \mathbb{P}^n\backslash B$ has a connected pre-image.
\end{remark}

Because the Lagrangian fibration from the previous proposition covers most of the symplectic manifold $X$, we expect it to carry non-trivial Floer theoretic data to probe the mirror of $X$. The next section is dedicated to computing the super-potential associated with this Lagrangian fibration.

\section{Computation of the super-potential}

Throughout this section, we will use homogeneous coordinates $[x_0:\dots:x_n]$ on projective space $\mathbb{P}^n$. Recall that the toric boundary for the classical action of $(\mathbb{C}^*)^n$ on projective space has the following defining equation:
\begin{equation*}
f_0=x_0\dotsi x_n.    
\end{equation*}
Let $f$ be a homogeneous polynomial of degree $n+1$ that is a \emph{generic small} perturbation of $f_0$, so that its zero locus $V(f)\subseteq \mathbb{P}^n$ is a smooth Calabi-Yau hypersurface. With such $f$, we associate the hypersurface:
\begin{equation*}
    X_f = V(t^{n+1}-f)\subseteq \mathbb{P}^{n+1},
\end{equation*}
which is an index $1$ Fano hypersurface of dimension $n$.

Projecting away from the point $[1:0:\dotsi:0]\in\mathbb{P}^{n+1}$, onto the hyperplane $H=\{t=0\}\cong \mathbb{P}^n$, produces a cyclic covering map:
\begin{equation*}
\phi:X_f\rightarrow \mathbb{P}^n,  
\end{equation*}
branched along the Calabi-Yau hypersurface $V(f)\subseteq \mathbb{P}^n$. In the previous section, we have shown that $X_f$ carries a Lagrangian torus fibration away from the ramification locus of $\phi$, see Proposition \ref{partial-fibration} . We now count (pseudo-)holomorphic discs of Maslov index $2$, with boundary on a Lagrangian torus fiber $L$, and passing through a fixed point in $L$. This count does not actually depend on the torus fiber. We therefore choose $L$ to be the monotone torus fiber, which in the context of Proposition \ref{partial-fibration}, arises as the pre-image of the Clifford torus $L_{\cl}\subseteq\mathbb{P}^n$:
\begin{equation*}
    L = \phi^{-1}(L_\cl).
\end{equation*}
 The strategy is to count curves in $\mathbb{P}^n$ with a prescribed tangency to the hypersurface $V(f)$; which is the branch locus of $\phi$. We will need to perturb the polynomial $f$, so let $\mathcal{H}$ be the vector space of homogeneous polynomials of degree $n+1$.
\subsection{Discs with tangency condition}
We start by fixing a relative homology class $\alpha\in H_2(\mathbb{P}^n,L_{\text{cl}})$. This class is determined by the intersection numbers:
\begin{equation*}
    \alpha_k = \alpha \cdot (x_k=0).
\end{equation*}
These numbers determine the Maslov index of $\alpha$ through the equation:
\begin{equation*}
    \frac{1}{2}\mu(\alpha) = \alpha_0 + \dotsi + \alpha_n.
\end{equation*}
We now recall a description of the space $\matheu{M}(L_{\text{cl}},\alpha)$ of holomorphic discs $v:(D,\partial D)\rightarrow (\mathbb{P}^n,L_{\text{cl}})$ in the class $\alpha$.
\begin{lemma} \label{coordinate-wise-description-of-v}
For each $v\in \matheu{M}(L_{\text{cl}},\alpha)$,
there exist holomorphic maps $v_k:(D,\partial D)\rightarrow (D,\partial D)$ of degree $\alpha_k$ such that:
\begin{equation*}
    v(z) = [v_0(z):v_1(z):\dots:v_n(z)].
\end{equation*}
\end{lemma}
\begin{proof}
The claim on degrees is automatic once we have the required description of $v$ in homogeneous coordinates, refer to Lemma \ref{decomposition-into-mobius-transformations} below. Because $v$ intersects the hyperplane $\{x_0=0\}$ in a finite subset $A_0\subset D$, we can write:
\begin{equation}\label{meromorphic-decomp}
    v(z) = [1:g_1(z):\dots:g_n(z)],
\end{equation}
where $g_k$ are holomorphic functions, possibly with singularities at the points of $A_0$. It suffices to show that these singularities, if they arise, are at worst poles. Let $z_{c}\in A_0$ be a singularity for $g_1$. By definition, this means that $v(z_c)$ belongs to the hyperplane $\{x_0=0\}$. But since the hyperplanes $(x_i=0)$ are linearly independent, one of them shouldn't contain $v(z_c)$. Without loss of generality, assume $v(z_c) \not\in \{x_1=0\}$. Similarly to (\ref{meromorphic-decomp}), we can then use an expression of $v$ in the complement of the hyperplane $\{x_1=0\}$ :
\begin{equation*}
    v(z) = [h_0(z):1:\dots:h_n(z)],
\end{equation*}
where the functions $h_k$ are holomorphic outside of a subset $A_1\subset D$, corresponding to the intersection of $v$ with $\{x_1=0\}$. In particular, $h_0$ is holomorphic near $z_c$ and $g_1 = 1/h_0$. It follows that $g_1$ is a meromorphic function as claimed. The same arguments applies to the remaining $g_2,g_3,\dots,g_n$.
\end{proof}

Recall the results of Cho-Oh in \cite{Cho-Oh-toric-regularity}, showing that the moduli space $\matheu{M}(L_{\text{cl}},\alpha)$ (for the integrable complex structure of $\mathbb{P}^n$) is Fredholm regular. Its dimension is computed using the Riemann-Roch formula:
\begin{equation*}
    \dim \matheu{M}(L_{\text{cl}},\alpha) = n + \mu(\alpha).
\end{equation*}
It can also be verified in this case using Lemmas \ref{coordinate-wise-description-of-v} and \ref{decomposition-into-mobius-transformations}. We will always assume $\alpha_k\geq 0$ for $k=0,1,\dots,n$. Otherwise, the corresponding moduli space is empty for the integrable complex structure.

Let us now introduce the relevant tangency moduli space. For each homogeneous polynomial of $f\in\mathcal{H}$ of degree $n+1$, we define:
\begin{equation}\label{tangency-condition}
    \tau^f_{\alpha} = \{ (v,z_0)\in \matheu{M}(L_{\text{cl}},\alpha)\times D\ |\ \ j_{n,z_0}(f\circ v) = 0\},
\end{equation}
where:
\begin{equation}\label{point-wise-jet-map}
    j_{n,z_0}(h) = (h(z_0),h'(z_0),\dots,h^{(n)}(z_0)).
\end{equation}
The tangency condition (\ref{tangency-condition}) imposes a minimal Maslov number constraint, when $f$ is near $f_0$. Indeed:
\begin{align*}
 \mu(\alpha) &= 2v\cdot V(f_0)\\
             &= 2v\cdot V(f)\\
             &\geq 2(n+1).
\end{align*}
More importantly, it means that $v$ comes from a holomorphic disc in the branched covering $X_f$, see Lemma \ref{lifting} below.

In our counting problems, we will always assume: 
\begin{equation}\label{minimal-maslov-number}
    \mu(\alpha)=2(n+1).
\end{equation}

We want to count elements of (\ref{tangency-condition}) with $1$ boundary constraint. To that end, we define:
\begin{equation*}
    \hat{\tau}^f_{\alpha,1} = \tau^f_{\alpha}\times \partial D /Aut(D).
\end{equation*}
It comes with a boundary evaluation map
\begin{equation}\label{evaluation map}
    ev: \hat{\tau}^f_{\alpha,1} \rightarrow L_{\text{cl}}.
\end{equation}
Ideally, the space $\hat{\tau}_{\alpha,1}^f$ will be an oriented closed manifold so that one can compute the degree $n_{\alpha}$ of the evaluation map. This is not always strictly true, and the goal of the remainder of this section is to highlight and resolve the difficulties that arise.

\subsubsection{The spherical class}
One of the relative homology classes with Maslov number $2(n+1)$ is actually spherical:
\begin{equation*}
    \alpha_s = (1,\dots,1).
\end{equation*}
The class $\alpha_s$ behaves somewhat differently from all the others and so we treat it separately.
\begin{proposition}\label{spherical-class-moduli-empty}
If the homogenous polynomial $f$ is sufficiently close to $f_0$, the moduli space $\tau^f_{\alpha_s}$ is empty.
\end{proposition}

\begin{proof}
Suppose we have a sequence $f_i\rightarrow f_0$, and elements $(v_i,z_i)\in \tau_{\alpha_s}^{f_i}$. Up to composition with M\"obius transformations, We may assume $z_i=0$.  By Gromov compactness, the sequence $v_i$ sub-converges to a genus $0$ nodal curve with boundary on $L_{\text{cl}}$. This limit must be tangent to the toric boundary $f_0$ to order $n$. Let $v_{\infty}$ be the component of this nodal curve that is tangent to $f_0$. Since $\mu(\alpha_s) = 2(n+1)$ is the minimal Maslov number for this order of tangency, all other components must in fact be constant. Now, the irreducible component $v_{\infty}$ is either a genuine disc with boundary on $L_{\text{cl}}$, or a projective line arising as a spherical bubble in the Gromov limit.

The former case can be ruled out using our description of discs in terms of M\"obius transformations:
\begin{equation*}
    v_{\infty} = [\phi_0:\dots:\phi_n],
\end{equation*}
since $f_0\circ v_{\infty}=\phi_0\dots\phi_n$ would have to be a degree $n+1$ disc endomorphism that vanishes at $0$ to order $n+1$, as implied by the tangency condition. This can only happen if all $\phi_k$ are multiples of $z$, and in this cases $v_{\infty}$ would be constant, which is a contradiction (see Lemma \ref{decomposition-into-mobius-transformations} below as well).

In the later case, $v_{\infty}$ should be a projective line tangent to the toric divisor to order $n$. This means that it is a line that intersects all components of the toric divisor simultaneously. But linear algebra rules this out, because these components are all linearly independent hyperplanes.
\end{proof}

In the remainder of this section, we present a systematic method to compute $n_{\alpha}$ for all non-spherical classes. From now on, $\alpha\neq \alpha_s$ is a relative homology class for the pair $(\mathbb{P}^n,L_\cl)$, whose Maslov index equals $2(n+1)$.

\subsection{Compactness and counting}
Let $\mathscr{E}_d(D)$ be the space of degree $d$ maps $(D,\partial D)\rightarrow (D,\partial D)$. Note that we are referring here to the topological degree of $v$, which can be computed for example from the pullback:
\begin{equation*}
v^*:H^1(\partial D,\mathbb{Z})\rightarrow H^1(\partial D,\mathbb{Z}),  
\end{equation*}
or equivalently using the integral formula:
\begin{equation*}
    \deg(v) = \int_{\partial D} v^*(d\theta).
\end{equation*}

\begin{lemma}\label{decomposition-into-mobius-transformations}
Any element $v\in \mathscr{E}_{d}(D)$ is a product of $d$ M\"obius transformations:
\begin{equation*}
    v(z) = \xi\prod_{k=1}^d\left( \frac{z-a_k}{1-\overline{a_k}z}\right),
\end{equation*}
where  $\xi$ is a unitary complex number and $a_k \in \text{int}(D)$, for $k=1,\dots,d$. The complex numbers $(a_k)$ will often be referred to as M\"obius centers.
\end{lemma}

\begin{proof}
We can prove this result by induction on $d$. Because $v$ is holomorphic, the topological degree formula above simplifies to:
\begin{equation*}
    \deg(v) = \frac{1}{2i\pi} \int_{\partial D}\frac{v'(z)}{v(z)}dz.
\end{equation*}
When $\deg(v)=0$, the argument principle then implies that $v(z)=0$ has no solutions. We claim now that $v$ must be constant. Indeed, if it weren't, the open mapping theorem would imply that $v(D)$ is an open subset of $D$. But $D$ is compact, so $v(D)$ should also be closed and so $v(D)=D$; that's a contradiction.

The induction step goes as follows: Given $v$ of degree $d\geq 1$, the argument principle implies that there exists $a\in D$ such that $v(a)=0$. We now pre-compose $v$ with the inverse $\phi^{-1}$ of the M\"obius transformation:
\begin{equation*}
    \phi(z) = \frac{z-a}{1-\overline{a}z}.
\end{equation*}
The result is a disc endomorphism $g = v\circ \phi^{-1}$ with the property that $g(0)=0$. Therefore, there exists a holomorphic function $h:D\rightarrow \mathbb{C}$ such that $g(z)=zh(z)$.
Note $h$ still restricts to a map $h:\partial D\rightarrow \partial D$. Since $D$ is compact, $h(D)$ is compact, and by the maximum principle we have $\partial h(D)\subseteq h(\partial D)$. It follows that $h$ is again a disc endomorphism whose degree is $d-1$.
\end{proof}

One can extract a set of global coordinates on the space $\mathscr{E}_d(D)$ from the previous Lemma: The set of elementary symmetric polynomials on the \emph{M\"obius centers} $(a_k)$, together with the angular coordinate $\xi$. This makes it easier to study the tangency equation in (\ref{tangency-condition}).

For example, one can show through direct computation that $\tau^{f}_{\alpha}$ (generally speaking) is not necessarily regular. Not even if we allow perturbations of $f\in\mathcal{H}$ near $f_0$. In other words, $0$ is not a regular value of the jet map:
\begin{align*}
    \matheu{M}(L_\cl,\alpha)\times D\times \mathcal{H} &\rightarrow \mathbb{C}^{n+1}\\
    (v,z_0,f)&\mapsto j_{n,z_0}(f\circ v).
\end{align*}

Nonetheless, it is still possible to calculate $n_{\alpha}$, if we interpret it to be the degree of the map:
\begin{align}\label{jet-map-with-perturbation}
    \Phi: \matheu{M}(L_\cl,\alpha)\times \mathcal{H}&\rightarrow \mathbb{C}^{n+1}\times \mathcal{H}\times L\\
    (v,f)&\mapsto \left(j_{n,0}(f\circ v), f, v(1)\right).\notag
\end{align}
Indeed, if one fixes a point $p\in L$, then the pre-image $\Phi^{-1}(0,f,p)$ counts holomorphic discs $v:(D,\partial D)\rightarrow (\mathbb{P}^n,L_\cl)$ in the homology class $\alpha$, that are tangent to $V(f)\subseteq\mathbb{P}^n$ to order $n$ at $z=0$, and such that $v(1)=p$. This fiber is essentially the same as $ev^{-1}(p)$ from (\ref{evaluation map}); the only difference is that we are taking a slice of the action of the automorphism group $\text{Aut}(D)$ by choosing $z=0$ to be the tangency point with $V(f)$, and $z=1$ to be the boundary marked point.

\begin{remark}\label{jet map ambiguity}
There is an ambiguity in the definition of the jet map $j_{n,0}$ from (\ref{jet-map-with-perturbation}): it depends on the choice of a representation of the holomorphic disc $v$ in homogeneous coordinates. However, when a class $\alpha\neq \alpha_s$ satisfying (\ref{minimal-maslov-number}) is fixed, there is a systematic way to produce such representations across the moduli space $\matheu{M}(L_\cl,\alpha)$. The reason is that for some $0\leq i\leq n$, we have vanishing of the intersection number:
\begin{equation*}
\alpha_i=(x_i=0)\cdot v = 0.    
\end{equation*}
Therefore, all holomorphic discs in the moduli space $\matheu{M}(L_\cl,\alpha)$ actually land in the open set $\mathbb{P}^n\backslash \{x_i=0\} = \{x_i=1\}$.
\end{remark}

If we want to ensure that the map in (\ref{jet-map-with-perturbation}) has a well defined degree, we need the following compactness result.
\begin{lemma}\label{properness}
There is an open neighborhood $U$ of $\{0\}\times\{f_0\}\times L$ such that the restriction of $\Phi$ to $U$ is a proper map of smooth manifolds.
\end{lemma}

\begin{proof}
Because $L$ is compact, the only potential cause of non-properness of the map $\Phi$ is the non-compact space $\matheu{M}(L_\cl,\alpha)$. We can remedy this by using the previously established relationship between this space and M\"obius transformations. First of all, using Lemma \ref{decomposition-into-mobius-transformations} and the remark thereafter, we can endow $\matheu{M}(L_\cl,\alpha)$ with smooth coordinates using the following parametrization:
\begin{align}\label{from moduli to mobius}
    \mathscr{E}_{\alpha_0,1}(D)\times\mathscr{E}_{\alpha_1}(D)\dots\times\mathscr{E}_{\alpha_n}(D)&\rightarrow \matheu{M}(L_\cl,\alpha)\\
    (v_0,v_1,\dots,v_n)&\mapsto [v_0:v_1:\dots:v_n],\notag
\end{align}
where:
\begin{equation*}
    \mathscr{E}_{\alpha_0,1}(D) = \left\{v\in\mathscr{E}_{\alpha_0}(D) | \ \ v(1)=1   \right\}.
\end{equation*}
This allows us to compactify $\matheu{M}(L_\cl,\alpha)$ by allowing the M\"obius centers $a_k$ from Lemma \ref{decomposition-into-mobius-transformations} to reach the boundary $\partial D$. We will denote the resulting compactification by $\overline{\matheu{M}}(L_\cl,\alpha)$. Note that discs in the boundary have strictly smaller Maslov numbers.

With this set-up in mind, we can prove the Lemma by way of contradiction. If it weren't true, there would exist an unbounded sequence $v_i\in \matheu{M}(L_\cl,\alpha)$ and $f_i\in \mathcal{H}$ such that:
\begin{equation*}
    f_i\rightarrow f_0 \ \ \text{and}\ \ j_{n,0}(f_i\circ v_i)=0.
\end{equation*}
After possibly passing to a sub-sequence, the maps $v_i$ will converge to an element $v_{\infty}$ of the boundary of $\overline{\matheu{M}}(L_\cl,\alpha)$, and we would still have the tangency equation:
\begin{equation*}
    j_{n,0}(f_0\circ v_{\infty})=0.
\end{equation*}
But since $\mu(v_{\infty})<2(n+1)$, the disc map $f_0\circ v_{\infty} : (D,\partial D) \rightarrow (D,\partial D)$ is non-constant and has degree at most $n$, and as such, it cannot vanish at $0$ to order $n$.
\end{proof}

\begin{remark}
This proof can also be rephrased using Gromov compactness, and then tracking the tangency point in the Gromov limit of the sequence $v_i$, in a similar spirit to the proof of Proposition \ref{spherical-class-moduli-empty}.
\end{remark}

For the purpose of studying the degree of $\Phi$, we recall some useful computational tools from differential topology.
\begin{lemma}\label{degree from submanifold}
Let $f:X\rightarrow Y$ be a proper smooth map between smooth oriented manifolds of the same dimension. Next, let $Z\subseteq Y$ be a smooth oriented submanifold that is transverse to $f$. Then $f^{-1}(Z)$ is a smooth oriented manifold and the degree of $f$ agrees with that of its restriction $f\big|_Z:f^{-1}(Z)\rightarrow Z$.
\end{lemma}

\begin{proof}
 Just observe that the transversality condition ensures that when $z\in Z$ is a regular value of $f\big|_Z$, it is also a regular value of $f$.
\end{proof}

\begin{lemma}\label{degree through cobordism}
Suppose that $\mathfrak{X}$ is an oriented cobordism between $X_0$ and $X_1$, and let $F:\mathfrak{X}\rightarrow Y$ be a proper smooth map to an oriented smooth manifold $Y$. Then the degrees of the restrictions of $F$ to either of its boundary components $X_i$ are the same.
\end{lemma}
\begin{proof}
Refer to Lemma 1 in chapter 5 of Milnor's book \cite{Milnor-diff-top}.
\end{proof}

We can now compute the desired degree.
\begin{lemma}\label{degree calculated}
If the relative homology class $\alpha$ has Maslov number $2(n+1)$, is different from $\alpha_s$, and has the component-wise decomposition $\alpha=(\alpha_0,\dots,\alpha_n)$, then:
\begin{equation*}
    \deg(\Phi) = \frac{(n+1)!}{\alpha_0!\dots\alpha_n!}.
\end{equation*}
\end{lemma}

\begin{proof}
We start by applying Lemma \ref{degree from submanifold} to restrict $\Phi$ to the submanifold $\mathbb{C}^{n+1}\times \{f_0\}\times \{p\}$, where $p=[1:\dots:1] \in L_\cl$. The required transversality condition follows from the fact that the evaluation map:
\begin{align*}
    \matheu{M}(L_{\cl},\alpha)&\rightarrow L_\cl\\
    v&\rightarrow v(1),
\end{align*}
is a submersion; refer again to Lemmas \ref{decomposition-into-mobius-transformations} and \ref{coordinate-wise-description-of-v}. Therefore, we may compute $\deg(\Phi)$ from the map:
\begin{align*}
   \hat{\Phi}: \matheu{M}_{p}(L_\cl,\alpha) &\rightarrow \mathbb{C}^{n+1}\\
    v &\mapsto j_{n,0}(f_0\circ v),
\end{align*}
where:
\begin{equation*}
     \matheu{M}_{p}(L_\cl,\alpha) = \left\{v\in  \matheu{M}(L_\cl,\alpha) \ \ |\ \ v(1)=p \right\}.
\end{equation*}
Referring back to Lemma \ref{decomposition-into-mobius-transformations} again, notice that we have a product decomposition:
\begin{equation}\label{M_{x_0} components}
     \matheu{M}_{p}(L_\cl,\alpha) = \mathscr{E}_{\alpha_0,1}(D)\times\dots\times \mathscr{E}_{\alpha_n,1}(D),
\end{equation}
using the same parametrization as (\ref{from moduli to mobius}). Moreover, the map $\hat{\Phi}$ is the composition of a product map and a jet map:
\begin{align*}
  \mathscr{E}_{\alpha_0,1}(D)\times\dots\times \mathscr{E}_{\alpha_n,1}(D) &\xrightarrow{\pi_\alpha} \mathscr{E}_{n+1,1} (D) \xrightarrow{j_{n,0}} \mathbb{C}^{n+1}\\
  v=(v_0,\dots,v_n)&\mapsto v_0\cdot v_1\dotsi v_n \mapsto j_{n,0}(\pi_{\alpha}(v)).
\end{align*}
Referring back to Lemma \ref{decomposition-into-mobius-transformations}, we can see that $\deg(\pi_\alpha)$ is the number of ways to partition a set $S$ of $n+1$ M\"obius transformations, into $n+1$-sets $S_i$, such that the size of $S_i$ is $\alpha_i$. This partition number is precisely:
\begin{equation*}
    \deg(\pi_{\alpha}) = \frac{(n+1)!}{\alpha_0!\dots\alpha_n!}.
\end{equation*}
It remains to show that the jet map:
\begin{align*}
    j_{n,0}: \mathscr{E}_{n+1,1}(D) &\rightarrow \mathbb{C}^{n+1}\\
    f&\mapsto j_{n,0}(f).
\end{align*}
has degree $1$. Again going back to (the proof of) Lemma \ref{decomposition-into-mobius-transformations}, we can see that $j_{n,0}^{-1}(0) = \{z^{n+1}\}$. Therefore, we need to show that $0$ is a regular value of this map.

\textit{Claim:} For each $d\geq 1$, $0$ is a regular value of the jet map:
\begin{equation*}
    j_{0,d}:\mathscr{E}_{1,d}(D)\rightarrow \mathbb{C}^d.
\end{equation*}

\textit{Proof of the claim.} This is a direct computation. An element $v\in \mathscr{E}_{1,d}(D)$, according to Lemma \ref{decomposition-into-mobius-transformations}, must have the form:
\begin{equation*}
    v(z) = \prod_{k=1}^d\left(\frac{1-\overline{a}_k}{1-a_k} \right) \prod_{k=1}^d\left(\frac{z-a_k}{1-\overline{a}_kz}\right).
\end{equation*}
As we have alluded to before, the elementary symmetric polynomials on $(a_1,\dots,a_k)$ provide us with a complete set of coordinates on the space $\mathscr{E}_{1,d}(D)$. Let $\lambda_i$ be the elementary symmetric polynomial of degree $i$. Then:
\begin{equation*}
    \prod_{d=1}^d(z-a_k)= \sum_{i=0}^d (-1)^i\lambda_i z^{d-i},
\end{equation*}
where we have set $\lambda_0=1$ for consistency of notation. If we now think of each $v\in\mathscr{E}_{1,d}(D)$ in terms of the coordinate $\lambda=(\lambda_1,\dots,\lambda_d)$ that defines it, we see that the jet map takes the form:
\begin{equation*}
    \lambda \mapsto j_{d,0}\left( \frac{R(\overline{\lambda},1)}{R(\overline{\lambda},z)}\cdot \frac{P(\lambda,z)}{P(\lambda,1)}\right),
\end{equation*}
where:
\begin{equation*}
    P(\lambda,z) = \sum_{i=0}^d (-1)^i\lambda_i z^{d-i} \ \ \text{and}\ \ R(\overline{\lambda},z) = \sum_{i=0}^d (-1)^i\overline{\lambda}_iz^i.
\end{equation*}
Our claim then says that $\lambda=0$ should be a regular point. Since we described the domain with complex coordinates, we find it easier to compute complex derivatives, even though $j_{d,0}$ is not holomorphic.
The derivatives at $\lambda=0$ are:
\begin{align*}
    (dj_{d,0})_{\lambda=0}\left(\partial_{\overline{i}}\right)&=0,\\
    (dj_{d,0})_{\lambda=0}(\partial_i) &= (0,\dots,0,(-1)^i,0,\dots,0),
\end{align*}
where $(-1)^i$ takes the $(d-i)^{\text{th}}$ slot. It follows that $0\in \mathbb{C}^d$ is indeed a regular value of $j_{d,0}$, which proves the claim, and the lemma as a consequence.
\end{proof}

\begin{remark}
One might ask what goes wrong in the proof of the previous lemma when $\alpha=\alpha_s$ is the spherical class that we have treated separately before. The main difference is that we no longer have the parametrization in (\ref{M_{x_0} components}) due to cancellations that occur in homogeneous coordinates when all the M\"obius transformations are equal to each other.
\end{remark}
\subsection{Transversality} Now that we have our numbers $n_{\alpha}$, we need to justify that they in fact correspond to generic counts of Maslov index $2$ discs in the branched covering $(X_f,L)$.

Let us fix a homology class $\beta\in H_2(X_f,L)$ of Maslov index $2$, and consider its corresponding moduli space:
\begin{equation*}
    \matheu{M}(L,\beta) = \left\{u:(D,\partial D)\rightarrow (X_f,L) \ \ |\ \ [u]=\beta \ \text{and}\ \overline{\partial}_J u=0 \right\}.
\end{equation*}
where $J$ is the integrable complex structure of $X_f$. We can then form the moduli space of unparametrized discs with $1$ boundary marked point:
\begin{equation*}
    \hat{\matheu{M}}_1(L,\beta) = \frac{\matheu{M}(L,\beta)\times \partial D}{\text{Aut}(D)}
\end{equation*}
and our goal is to compute the degree $m_{0,\beta}(L)$ of the evaluation map:
\begin{align*}
    ev:\hat{\matheu{M}}_1(L,\beta)  &\rightarrow L\\
    (u,e^{i\theta})&\mapsto u(e^{i\theta}).
\end{align*}
The way it was defined above, the moduli space $\hat{\matheu{M}}_1(L,\beta)$ is not always regular, and we  will need to introduce Hamiltonian perturbations to achieve transversality. This detail will be revisited soon.

Because $\mu(\beta)=2$, each holomorphic disc $u$ will intersect the ramification divisor $\{t=0\}\subseteq X_f$ exactly once. We can use this idea to take a slice of the action of the group $\text{Aut}(D)$ above: Fix the boundary point to be $1$, and the intersection point with $R$ to be $u(0)$. In other words, $m_{0,\beta}(L)$ is also the degree of the map:
\begin{align*}
    ev_1:\hat{\matheu{M}}_0(L,\beta)  &\rightarrow L\\
    u &\mapsto u(1),
\end{align*}
where:
\begin{equation}\label{moduli-space-in-cover}
    \hat{\matheu{M}}_0(L,\beta) =\left\{u\in \matheu{M}(L,\beta) \ | \  t(u(0))=0  \right\}.
\end{equation}
The plan is to compute the previous degree by pushing this whole story down to $\mathbb{P}^n$ using the branch covering map:
\begin{equation*}
    \phi : X_f\rightarrow \mathbb{P}^n.
\end{equation*}
We recall that $f$ was a degree $n+1$ homogeneous polynomial in the coordinates $[x_0:\dots:x_n]$ of projective space; that it was chosen to be sufficiently close to $f_0=x_0\cdot x_1\dotsi x_n$; and that $X_f=V(t^{n+1}-f)\subseteq \mathbb{P}^{n+1}$. The map $\phi$ is "forgetting" the $t$-coordinate. It is branched along $B_f=V(f)\subseteq \mathbb{P}^n$, and ramified along $R_f=\{t=0\} = \phi^{-1}(B_f)\subseteq X_f$. For technical transversality reasons, we will need to work with an \emph{enlarged} ramification locus:
\begin{equation*}
R_+ = \phi^{-1}(B_f\cup V(f_0)). 
\end{equation*}

In order to carry out this plan, we need the moduli space $\hat{\matheu{M}}_0(L,\beta)$ to be regular. This can be achieved by keeping $J$ fixed, and perturbing the $J$-holomorphic equation to:
\begin{equation*}
    (du-Y)^{0,1} = 0,
\end{equation*}
where the perturbation datum $Y \in \Omega^1(D,\Gamma_0(TX_f))$ is a $1$-form on $D$ with values in the space of Hamiltonian vector fields on $X_f$, that have compact support in the complement of $R_+$ in $X_f$. This class of perturbations is large enough to achieve transversality, see for example \cite{PL-theory}, section (9k). We can therefore fix a sufficiently small $Y$ for which the perturbed moduli space:
\begin{align*}
    \hat{\matheu{M}}^Y_0(L,\beta) =\{u:(D, \partial D)\rightarrow  (X_f,L) \ | \  (du-Y)^{0,1} = 0,&\ \\ \text{and}\  t(u(0))=0, \ [u]=\beta &  \},
\end{align*}
is a regular. Next, we pushforward this whole story to $\mathbb{P}^n$. Let $\alpha=\phi_{*}(\beta)$, and we set $Z=\phi_*Y \in \Omega^1(D,\Gamma_0(T\mathbb{P}^n))$ to be the pushforward of the perturbation datum $Y$, where now $\Gamma_0(T\mathbb{P}^n)$ stands for vector fields with compact support in the complement of the branch locus $V(f)\subseteq \mathbb{P}^n$.

Our choice of perturbation data ensures that an element $u\in\hat{\matheu{M}}^Y_0(L,\beta) $ is genuinely holomorphic near the ramification locus, and therefore, its pushforward $v=u\circ\phi$ is also holomorphic near the branch locus $V_f$, Moreover:
\begin{equation}\label{tangency-equation}
    j_{n,0}(f\circ v) = 0.
\end{equation}
This pushforward $u\mapsto v$ is an \emph{unbranched} covering map of degree $n+1$. To see that, we consider the perturbed tangency moduli space:
\begin{align*}
  \hat{\tau}^Z_0(L_{\cl},\alpha) =\{v:(D, \partial D)\rightarrow (\mathbb{P}^n,L_{\cl}) \ | \  (dv-Z)^{0,1} = 0,&\\ \text{and}\ j_{n,0}(f\circ v)=0, \ [v]=\alpha & \}.
\end{align*}

\begin{lemma}\label{lifting}
Every disc map $v\in \hat{\tau}^Z_0(L_{\cl},\alpha)$ has $n+1$ distinct lifts $(u^i)_{0\leq i\leq n}\in \hat{\matheu{M}}^Y_0(L,\phi^{-1}(\alpha))$, where:
\begin{equation*}
  \hat{\matheu{M}}^Y_0(L,\phi^{-1}(\alpha)) = \bigcup_{\beta\in \phi^{-1}(\alpha)} \hat{\matheu{M}}^Y_0(L,\beta) .
\end{equation*}
Furthermore,
\begin{equation}
   \sum_{\beta\in \phi^{-1}(\alpha)} m_{0,\beta}(L) = \deg\left( \hat{\tau}^Z_0(L_{\cl},\alpha) \xrightarrow{ev_1} L_{\text{cl}}\right).
\end{equation}
\end{lemma}

\begin{proof}
By abuse of notation, we will identify $v$ with its component-wise description in homogeneous coordinates in order to study the composition $f\circ v : D\rightarrow \mathbb{C}$, the ambiguity in this choice is the subject of Remark \ref{jet map ambiguity}. In a small disc $\{\abs{z}<r\}$, this function is holomorphic and vanishes to degree $n+1$ at $0$, and hence there exists a function $t:\{ \abs{z}<r\} \rightarrow \mathbb{C}$ such that:
\begin{equation*}
    t(z)^{n+1} = f\circ v(z).
\end{equation*}

This produces a lift $u=[t:v]$ of $v$, but only on the smaller domain $t:\{ \abs{z}<r\}\rightarrow\mathbb{C}$.  Because $f\circ v$ only vanishes at $0$ (otherwise the Maslov number of $\alpha$ would be bigger than $2(n+1)$), the lift $u$ above can be extended to the whole of $D$ using the path lifting property of the unbranched covering map $X_f\backslash R\rightarrow \mathbb{P}^n\backslash V_f$. The number of lifts is $n+1$, because on $\{ \abs{z}<r\}$, the equation $t(z)^{n+1} = f\circ v(z)$ has exactly $n+1$ solutions in $t$. Moreover, the unique continuation principle for solutions of the perturbed $J$-holomorphic equation:
$$(du-Y)^{0,1} = 0,$$
ensures that two solutions $u_1$ and $u_2$ that agree on a non-empty open set, must in fact be identical. Finally, the degree formula follows from the diagram of covering spaces:
\begin{center}
\begin{tikzcd}
\hat{\matheu{M}}^Y_0(L,\phi^{-1}(\alpha))\arrow[d] \arrow[r, "ev_1"] & L \arrow[d]\\
\hat{\tau}^Z_0(L_{\cl},\alpha) \arrow[r,"ev_1"] & L_{\text{cl}}
\end{tikzcd}
\end{center}
because the vertical maps both have degree $n+1$.
\end{proof}

\begin{remark}
When $n>2$, the pushforward map $\phi_*: H_2(X_f,L)\rightarrow H_2(\mathbb{P}^n,L)$ is injective. This is because $H_2(X_f)$ is generated by a hyperplane section, which is fixed by Deck transformations of the branched covering map $\phi:(X_f,L)\rightarrow (\mathbb{P}^n,L)$. In particular, Deck transformations act trivially on $H_2(X_f,L)$. This property fails in dimension $2$. This is not an issue however, because when the Maslov index 2 classes $\beta_1,\beta_2\in H_2(X_f,L)$ are in the same orbit of the $\mathbb{Z}_3$-action, we in fact have $m_{0,\beta_1}(L)=m_{0,\beta_2}(L)$.
\end{remark}

The previous argument omits at least one important technical detail, and that is the regularity of the perturbed tangency space. In order to address this issue, we introduce the deformed moduli space:
\begin{equation*}
    \matheu{M}^Z(L_{cl},\alpha) = \left\{ v:(D, \partial D)\rightarrow (\mathbb{P}^n,L_{\cl}) \ | \ [v]=\alpha,\  (dv-Z)^{0,1} = 0 \right\}.
\end{equation*}
Note that $Z=0$ corresponds to the unperturbed moduli space of holomorphic discs with boundary on the Clifford torus in the class $\alpha$, and with the standard complex structure of projective space. Recall that this moduli space is Fredholm regular, see theorem 6.1 \cite{Cho-Oh-toric-regularity}. Since we are perturbing using a small $Y$ (and hence a small $Z=\phi_*(Y)$), this Fredholm regularity is not lost. On this moduli space, we can define a jet map:
\begin{align}\label{perturbed jet map}
    j^f_{n,0}: \matheu{M}^Z(L_{cl},\alpha) &\rightarrow \mathbb{C}^{n+1}\\
    v &\mapsto j_{n,0}(f\circ v).\notag
\end{align}
Just like in Remark \ref{jet map ambiguity}, there is an ambiguity in defining this map, but it is resolved by the same argument: Indeed, the perturbation datum $Z$ vanishes near the toric divisor $V(f_0)$ in projective space, and that implies that for any $v\in \matheu{M}^Z(L_{cl},\alpha)$, the disc $v$ intersects the hyperplane $(x_i=0)$ finitely many times, and that all the intersection points have positive contributions to the intersection number $v\cdot (x_i=0)$. But again, the Maslov number constraint (\ref{minimal-maslov-number}) (together with $\alpha\neq \alpha_s$) forces one of these intersection numbers to be $0$. In particular, we can fix an $i$ such that $\alpha_i=0$, and then all elements $v\in \matheu{M}^Z(L_{cl},\alpha)$ will have a unique representation in homogeneous coordinates where the $i^{\text{th}}$ coordinate is constantly equal to $1$, and this coordinate representation makes $j^f_{n,0}$ well defined.
\begin{remark}\label{convention coordinate equals 1}
From now on, we assume that we have fixed $i$ such that $\alpha_i=0$, so that all discs in the moduli space $\matheu{M}^Z(L_{cl},\alpha)$ have image in the open set $\{x_i=1\}$, and we think of $f$ as a function on this open set.
\end{remark}
We are now in position to state the main regularity theorem of this section.

\begin{proposition}\label{transversality}
In the jet map (\ref{perturbed jet map}), $0\in \mathbb{C}^{n+1}$ is a regular value.
\end{proposition}

\begin{proof}
Let $v\in (j^f_{n,0})^{-1}(0)$. By the work of Lemma \ref{lifting}, there exists $u\in \matheu{M}_0^Y(L,\beta)$ such that $\phi\circ u = v$. By differentiating the maps:
\begin{equation*}
 \matheu{M}_0^Y(L,\beta)\xrightarrow{\phi}   \matheu{M}^Z(L_{cl},\alpha)\xrightarrow{j^f_{n,0}} \mathbb{C}^{n+1},
\end{equation*}
we obtain a sequence of vector spaces:
\begin{equation}\label{regularity ses}
 T_u\matheu{M}_0^Y(L,\beta)\xrightarrow{\phi_*}   T_v\matheu{M}^Z(L_{cl},\alpha)\xrightarrow{(dj^f_{n,0})_v} \mathbb{C}^{n+1}.
\end{equation}
We prove regularity by showing that this is actually a short exact sequence, which we call the regularity sequence.

Recall that the tangent space $T_v\matheu{M}^Z(L_{cl},\alpha)$ is the kernel of the linearization of the perturbed $\overline{\partial}$-equation. This looks like:
\begin{equation*}
    D_v: \Gamma\left(D,v^*T\mathbb{P}^n, v_{\partial D}^* TL_{\cl}\right) \rightarrow \Omega^{0,1}(D,v^*T\mathbb{P}^n).
\end{equation*}
The same applies to $T_u\matheu{M}_0^Y(L,\beta)$, except that the constraint $t(u(0))=0$ restricts the domain of the linearized operator a bit:
\begin{equation*}
    D_u: \Gamma_0\left(D,u^*TX_f, u_{\partial D}^* TL\right) \rightarrow \Omega^{0,1}(D,u^*TX_f),
\end{equation*}
where:
\begin{equation*}
    \Gamma_0\left(D,u^*TX_f, u_{\partial D}^* TL\right) = \left\{ \xi\in \Gamma\left(D,u^*TX_f, u_{\partial D}^* TL\right) \ \ |\ \ \xi_0\in T_{u(0)}R_f\right\}.
\end{equation*}
The best way to understand the regularity sequence (\ref{regularity ses}) is by examining the sheafy versions of $\ker(D_u)$ and $\ker(D_v)$. For an open set $U\subseteq D$, let:
\begin{align*}
    \matheu{E}_u(U) &= \{\xi\in \Gamma\left(U,u^*TX_f, u_{\partial D}^* TL\right) \ \ |\ \ \xi_0\in T_{u(0)}R_f \ \ \text{and}\ \ D_u(\xi)=0\}\\
    \matheu{E}_v(U) &= \{\xi\in \Gamma\left(U,v^*T\mathbb{P}^n, v_{\partial D}^* TL_\cl\right) \ \ |\ \  D_v(\xi)=0\}.
\end{align*}
To obtain the regularity sequence, it suffices to show that we have a short exact sequence of sheaves:
\begin{equation}\label{sheaf ses}
    0\rightarrow \matheu{E}_u \xrightarrow{\phi_*} \matheu{E}_v\xrightarrow{dj^f_{n,0}} \underline{\mathbb{C}}^{n+1}\rightarrow 0,
\end{equation}
where $\underline{\mathbb{C}}^{n+1}$ is a skyscrapper sheaf at $0\in D$. Indeed, Fredholm regularity of $D_u$ means that $H^1(\matheu{E}_u)=0$ and so, by appealing to the long exact sequence in sheaf cohomology, we get the short exact sequence in (\ref{regularity ses}).

Looking at the sequence (\ref{sheaf ses}), observe that the map $\phi_*$  restricts to a sheaf isomorphism on the punctured disc $D\backslash 0$. It means that both kernel and cokernel are supported at $0$. Because of the identity principle for solutions of the Cauchy-Riemann equations:
\begin{equation*}
D_u(\xi)=0,   
\end{equation*}
we can already deduce that the sheaf map $\phi_*:\matheu{E}_u \rightarrow{} \matheu{E}_v$ is injective. 

Next, as the cokernel is supported at $0$, we can compute it by trivializing near $u(0)\in R_f$. In (resp. below) a small neighborhood of $u(0)$, the perturbation data vanishes and $D_u$ (resp. $D_v$) is the Dolbeaux operator. Furthermore, following the conventions of Remark \ref{convention coordinate equals 1}, $f$ defines a regular function in a neighborhood of $v(0)$, and the branched covering has the local model:
\begin{align*}
    (t^{n+1}-x_1)\subseteq \mathbb{C}^{n+1} \rightarrow \mathbb{C}^n \\
    (t,x_1,\dots,x_n)\mapsto (x_1,\dots,x_n).
\end{align*}
in which $u(0)=0\in \mathbb{C}^{n+1}$, $v(0)\in \mathbb{C}^n$ and $f=x_1$. In this local model, we have a trivializing frame for both $TX_f$ and $T\mathbb{P}^n$. The first is given by the vector fields $\partial_t,\partial_{x_2}$,...,$\partial_{x_n}$ and the later is given by $\partial_{x_1} ,\partial_{x_2}$,...,$\partial_{x_n}$. Moreover, the action of $\phi_*$ on this frame is:
\begin{equation*}
    \phi_*(\partial_t) = (n+1)t^n \partial_{x_1} \ \ \text{and}\ \ \phi_*(\partial_{x_k})=\partial_ {x_k} \ \text{for}\ 2\leq k \leq n.
\end{equation*}
The holomorphic disc $u$ has a coordinate description in this chart:
\begin{equation*}
    u(z)=(t(z),x_1(z),\dots,x_n(z)).
\end{equation*}
This is defined over a small open set $0\in U\subseteq D$. Moreover:
\begin{equation*}
    \matheu{E}_u(U) = \{f_0(z)\partial_t+f_2(z)\partial_{x_2}+\dotsi+f_n(z)\partial_{x_n}\ | \overline{\partial} f_i = 0 \ \ \text{and}\ \ f_0(0)=0\},
\end{equation*}
and at the same time:
\begin{equation*}
\matheu{E}_v(U) = \{f_1(z)\partial_{x_1}+f_2(z)\partial_{x_2}+\dotsi+f_n(z)\partial_{x_n}\ | \overline{\partial} f_i =0\}.
\end{equation*}
Furthermore, the jet map (recall the conventions of Remark \ref{convention coordinate equals 1}) has the formula:
\begin{align*}
    dj^f_{n,0}: \matheu{E}_v(U) &\rightarrow \mathbb{C}^{n+1}\\
    \sum_{k=1}^n f_k(z)\partial_{x_k} &\mapsto j_{n,0}(f_1(z)).
\end{align*}
Therefore, by taking stalks at $0\in D$ in the sequence (\ref{sheaf ses}), we reduce our transversality problem to the following short exact sequence:
\begin{equation}\label{ses-at-0}
0\rightarrow \{f\in  \mathscr{H}_0 \ |\ f(0)=0\}\xrightarrow{\times t(z)^n} \mathscr{H}_0 \xrightarrow{j_{n,0}} \mathbb{C}^{n+1} \rightarrow 0    ,
\end{equation}
where $\mathscr{H}_0$ is the stalk at $0$ of the sheaf of holomorphic functions on $D$.
The sequence in (\ref{ses-at-0}) above is exact because the vanishing order of $t$ at $0$ is exactly $1$, i.e $t(z)=z\cdot \epsilon$, where $\epsilon$ is an invertible element of $\mathscr{H}_0$.
\end{proof}

As a consequence of Proposition \ref{transversality}, of Lemma \ref{lifting}, and Lemma \ref{degree from submanifold}, we can actually see that $m_{0,\beta}(L)$ is the local degree of the map:
\begin{align}\label{wide jet map}
    \Psi_1: \matheu{M}^Z(L_{\cl},\alpha)\times\mathcal{H}\rightarrow \mathbb{C}^{n+1}\times \mathcal{H}\times L\\
    (v,f)\mapsto (j_{n,0}(f\circ v), f, v(1))\notag
\end{align}
near $\{0\}\times\{f_0\}\times L$ (recall that we only have properness in this region, see Lemma \ref{properness}). We now have all the necessary ingredients to prove the main result of this section.

\begin{theorem}\label{m0-numbers}
Let $X_f=V(t^{n+1}-f)$ be a smooth hypersurface of degree $n+1$, and let $\phi: X_f\rightarrow \mathbb{P}^n$ be the linear projection onto the hyperplane $\{t=0\}$. Let $L_{\text{cl}}\subset H$ be the Clifford torus. If $f$ is generic and nearly degenerate, then $L_{\text{cl}}$ lifts to a totally real torus $L$ in $X_f$. Moreover, counts of Maslov index $2$ discs with respect to an anti-canonical K\"ahler form are given by the formula:
\begin{equation*}
\sum_{\beta\in \phi^{-1}(\alpha)} m_{0,\beta}(L) = \frac{(n+1)!}{\alpha_0!\dots \alpha_n!},    
\end{equation*}
for any class $\alpha=(\alpha_0,\dots,\alpha_n)\in H_2(\mathbb{P}^n,L_\cl)$ of Maslov index $2(n+1)$, except when $\alpha=\alpha_s$ is the spherical class. In that case:
\begin{equation*}
    m_{0,\beta}(L)=0,
\end{equation*}
for all $\beta\in \phi^{-1}(\alpha_s)$.
\end{theorem}

\begin{proof}
The only part of the theorem above that we haven't proved yet is the degree formula when $\beta$ is not spherical. The key is that we can scale down the perturbation datum $Z$ by a real number $s\in [0,1]$, without losing regularity of the moduli space $\matheu{M}^{sZ}(L_\cl,\alpha)$, because $Z$ is small and $\matheu{M}^0(L_\cl,\alpha)$ is Fredholm regular. We can therefore deform the map (\ref{big jet map}) through a cobordism:
\begin{align}\label{big jet map}
    \Psi: \matheu{M}_{[0,1]}^Z(L_{\cl},\alpha)\times\mathcal{H}\rightarrow \mathbb{C}^{n+1}\times \mathcal{H}\times L\\
    (v,f)\mapsto (j_{n,0}(f\circ v), f, v(1)),\notag
\end{align}
where 
\begin{equation*}
\matheu{M}_{[0,1]}^Z(L_{\cl},\alpha) = \{ (v,s) \ | \ s\in [0,1] \ \text{and} \ v\in \matheu{M}^{sZ}(L_\cl,\alpha) \}.
\end{equation*}
By applying Lemma \ref{degree through cobordism}, we deduce that $m_{0,\beta}(L)$ agrees with the degree $n_{\alpha}$ of the jet map defined in (\ref{jet-map-with-perturbation}). Finally, the formula for $n_{\alpha}$ was obtained in Lemma \ref{degree calculated}, and the corresponding formula for $m_{0,\beta}(L)$ follows.
\end{proof}

\subsection{Super-potential} Recall that the super-potential associated with a Lagrangian torus $L\subseteq X_f$ is a function on its mirror space:
\begin{equation*}
    M_L = \text{Spec}(\mathbb{C}[H_1(L,\mathbb{Z})]).
\end{equation*}
For an Abelian group $A$, the algebra $\mathbb{C}[A]$ has a generator $z_a$ for each element $a\in A$, subject to the relation:
\begin{equation*}
    z_{a+a'} = z_az_a'.
\end{equation*}
The potential function is then given by the formula:
\begin{equation*}
    W = \sum_{\mu(\beta)=2} m_{0,\beta}(L)z_{\partial \beta}.
\end{equation*}
It is somewhat easier to write the potential function on the mirror of $L_{\text{cl}}$ first. By construction, loops in $H_1(L_{\text{cl}})$ lift to $L$ if and only if they link trivially around the toric boundary. In other words, if they lie in the kernel of the map:
\begin{align*}
    H_1(L_\cl) &\rightarrow \mathbb{Z}_{n+1}\\
    \gamma &\mapsto u_{\gamma}\cdot D_0,
\end{align*}
where $u_{\gamma}$ is any a disc whose boundary is $\gamma$, and $D_0=V(x_0\dotsi x_n)\subseteq \mathbb{P}^n$ is the toric boundary. As a consequence, we have a short-exact sequence:
\begin{equation*}
  0\rightarrow  H_1(L)\rightarrow H_1(L_{\text{cl}}) \rightarrow \mathbb{Z}_{n+1} \rightarrow 0.
\end{equation*}
Passing to group algebras, we obtain:
\begin{equation} \label{group algebra ses}
0\rightarrow  \mathbb{C}[H_1(L)]\rightarrow \mathbb{C}[H_1(L_{\text{cl}})] \rightarrow \mathbb{C}[z]/[z^{n+1}-1]  \rightarrow 0.
\end{equation}
This short exact sequence describes a cyclic $n+1$ covering map $\pi:M_{L_{\text{cl}}}\rightarrow M_L$.
The space $M_{L_{\text{cl}}}$ has natural coordinates coming from the paths $\gamma_k=\partial u_k$, where the discs $u_k$ are the generators that we defined in (\ref{generators of relative H2}). We therefore set $z_k = z_{\gamma_k}$, and we note that these elements satisfy the equation:
\begin{equation*}
    z_0\dots z_n = 1.
\end{equation*}
With this set of coordinates, we can compute the pullback of the potential-function using Theorem \ref{m0-numbers}:
\begin{equation}\label{pullback of W}
    \pi^*W = \left(z_0+\dots +z_n\right)^{n+1}-(n+1)!.
\end{equation}
To compute $W$ itself, we need a set of coordinates in $M_L$. This is not difficult once we understand that the quotient map in the short exact sequence (\ref{group algebra ses}) is:
\begin{align*}
    \mathbb{C}[H_1(L_{\text{cl}})] &\rightarrow \mathbb{C}[z]/[z^{n+1}-1]\\
    z_{\alpha} &\mapsto z^{\sigma(\alpha)},
\end{align*}
where $\sigma(\alpha) = \alpha_0 + \dots + \alpha_{n}$.
Therefore, by setting $y_k=z_k/z_0$, we produce a homomorphism of algebras:
\begin{equation*}
    \mathbb{C}[y_1,\dots,y_n] \rightarrow \mathbb{C}[H_1(L)].
\end{equation*}
Using the equation $y_1\dots y_n = z_0^{-(n+1)}$, we see that the morphism above is localization at the product $y_1\dots y_n$. As a conclusion of this analysis, we obtain the following result.

\begin{proposition}\label{formula-for-W}
There is an embedding of $M_L$ in $\mathbb{C}^{n}$ as the complement of the standard set of axes. In this coordinate system, the potential function is given by
\begin{equation} \label{potential-function}
    W = \frac{(1+y_1+\dots+y_n)^{n+1}}{y_1\dots y_n}-(n+1)!.
\end{equation}
\end{proposition}

\begin{remark}
The super-potential above agrees with Givental's Landau-Ginzburg model associated with $X_f$, which is typically computed from its Gromov-Witten invariants. See for instance \cite{Katzarkov-Przyj} for an overview, and section 3 of \cite{Przyj} for some explicit formulae. 
\end{remark}

The potential function $W$ has the expected critical values:
\begin{equation*}
    w_b = -(n+1)! \ \ \text{and}\ \ w_s=(n+1)^{n+1}-(n+1)!. 
\end{equation*}
These are the eigenvalues of quantum multiplication by $c_1$ on the cohomology ring of $X_f$. We call $w_s$ the small critical value; the fiber there has an isolated non-degenerate singularity, and we often call $w_s$ the non-degenerate critical value. We call $w_b$ the big critical value, and the fiber there is not reduced, but it's reduction is the smooth $(n-1)$-dimensional pair of pants.

Finally, we note that the relationship between $M_L$ and $M_{L_\text{cl}}$ runs even deeper. Indeed, the potential function $W_\cl$ for $L_{\text{cl}}$ is known in the literature. In the same set of coordinates used in equation (\ref{pullback of W}), this super-potential has the formula:
\begin{equation*}
    W_\cl = z_0+\dots+z_n.
\end{equation*}
With that in mind, We obtain a commutative diagram:
\begin{displaymath}
\begin{tikzcd}
M_{L_{\text{cl}}}\arrow{d}{\pi} \arrow{r}{W_{\text{cl}}} & \mathbb{C} \arrow{d}{z^{n+1}}\\
M_L\arrow{r}{\hat{W}} & \mathbb{C}
\end{tikzcd}
\end{displaymath}
where:
\begin{equation}
    \hat{W} = W + (n+1)!.
\end{equation}

We will explore this relationship in detail and use it to study homological mirror symmetry for the super-potential $W$.


\section{Generation of the small component}
The goal of this section is to prove that the monotone Lagrangian torus at the center of our (partial) SYZ fibration generates the small component of the Fukaya category. This section contains no new results, but is instead a compilation of all the ingredients needed to establish homological mirror symmetry over the small component.

\subsection{Monotone Floer theory, review} Let $(X,\omega)$ be a monotone symplectic manifold, such that $\omega$ an anti-canonical form, and let $L\subseteq X$ be a monotone Lagrangian brane. As far as we know, there are two main approaches to associating an $A_{\infty}$-algebra $A=CF(L)$ with $L$. One approach is to count holomorphic polygons with boundaries on small push-offs of $L$, following the same lines of \cite{PL-theory}; this method makes use of the monotonicity assumption to achieve transversality and compactness. Another more general approach, carried out by Fukaya-Oh-Ohta-Ono in \cite{FOOO}, relies on chain-level intersection theory in the moduli spaces $\matheu{M}_{d+1}(L)$ of discs with boundary on $L$ . The former approach is the one adopted by N. Sheridan in \cite{sheridan-fano-mirror-symmetry} and has its own advantages: It yields an $A_{\infty}$-algebra over $\mathbb{Z}$, and its underlying $\mathbb{Z}$-module is small, generated only by the intersection points of $L$ with one of its nearby perturbations. In our work however, we find it more convenient to work with the later approach, as it comes with a \emph{divisor axiom}, which makes our Floer cohomology computations easier. We therefore recall the main characteristics of this construction.

In \cite{Fukaya_cyclic}, K. Fukaya constructs an $A_{\infty}$-algebra structure $(m_k)_{k\geq 1}$ on the $\mathbb{Z}$-graded vector space:
\begin{equation*}
    A = H^*(L,\mathbb{C}[[q]]),
\end{equation*}
where $q$ is a formal parameter of degree $2$. Ignoring all analytical, topological and algebraic complications, the $A_{\infty}$-structure maps:
\begin{equation*}
m_{k}: A^{\otimes k}\rightarrow A[2-k],    
\end{equation*}
 have a sum decomposition:
\begin{equation}\label{structure maps}
    m_{k} = \sum_{\beta\in H_2(X,L)} q^{\langle \omega,\beta\rangle} m_{k,\beta},
\end{equation}
with respect to topological types $\beta\in H_2(X,L)$, and each term $m_{k,\beta}$
is a cohomological Fourier-Mukai transform based on the correspondence:
\begin{center}
\begin{tikzcd}\label{two-lfs}
& \overline{\matheu{M}_{k+1}(L,\beta)} \arrow[dl,"ev_k\times\dotsi\times ev_1"'] \arrow[dr,"ev_0"]  &\\
L^k&                 &L
\end{tikzcd}
\end{center}
where $\overline{\matheu{M}_{k+1}(L,\beta)}$ is the Gromov compactification of the space of holomorphic discs in the class $\beta$, with boundary on $L$ and carrying $k+1$ boundary marked points. The non-constant discs are responsible for terms of $m_k$ that involve non-constant powers of $q$, these are sometimes called instanton corrections. For the reader's convenience, we recall the index formula that shows which Maslov numbers are relevant in each term of $m_{k}$:
\begin{equation}\label{dimension-formula}
    \dim \matheu{M}_{k+1}(L,\beta) = k-2+n+\mu_L(\beta).
\end{equation}

In Lemma 13.2 of \cite{Fukaya_cyclic}, it is proved that $A$ is strictly unital and that the structure maps satisfy a divisor axiom: Given $b\in A^1$, an integer $k\geq 0$, elements $x_1,\dots,x_k \in A$, and $s\geq 0$ another integer then:
\begin{equation*}
    \sum_{s_0+\dots+s_k=s} m_{k+s,\beta}(b^{\otimes s_0},x_1,b^{\otimes s_1},\dots,x_k,b^{\otimes s_k}) = \frac{1}{s!}\left(\partial\beta\cap b\right)^s m_{k,\beta}(x_1,\dots,x_k).
\end{equation*}
Recall that when $k=0$, the element $m_{0,\beta}\in\mathbb{C}$ is simply the regular count of isolated holomorphic discs with boundary on $L$.

The $A_{\infty}$-algebra can be deformed using \emph{bounding co-chains}, which are elements $b\in H^1(L,\mathbb{C})$ for which we have an equation:
\begin{equation}\label{potential equation}
    m_1(b)+m_2(b,b)+\dots = p(b)\mathbf{1}_A,
\end{equation}
where $p(b)$ is an element of $\mathbb{C}[[q]]$. The $b$-deformed $A_{\infty}$-structure is given by the equation:
\begin{equation*}
    m^b_{k}(x_1,\dots,x_k) = \sum_{s_0+\dots+s_k=s} m_{k+s}(b^{\otimes s_0},x_1,b^{\otimes s_1},\dots,x_k,b^{\otimes s_k}).
\end{equation*}
In our setting, equation (\ref{potential equation}) holds automatically, and the assignment:
\begin{equation*}
p:H^1(L,\mathbb{C})\rightarrow \mathbb{C}[[q]],    
\end{equation*}
is called the \emph{potential function} of $L$.

\subsection{Monotone Floer theory, calculation} We now apply the general framework above to compute Fukaya's $A_{\infty}$-algebra associated with the monotone Lagrangian torus $L\subset X\backslash D$ constructed in Proposition \ref{partial-fibration}. Recall that:
\begin{equation*}
 X= V(t^{n+1}-f) \subseteq \mathbb{P}^{n+1},   
\end{equation*}
where $f$ is a generic homogeneous polynomial of degree $n+1$ in the variables $x_0,\dots,x_n$, sufficiently close to the product:
\begin{equation*}
f_0 = x_0\cdot x_1\dotsi x_n,    
\end{equation*}
which in turn is the defining equation of the toric boundary of $\mathbb{P}^n$. The index 1 Fano hypersurface $X$ above comes with a cyclic covering map (drop $t$):
\begin{equation*}
\phi: X\rightarrow \mathbb{P}^n,    
\end{equation*}
 branched over the zero locus of $f$. The appropriate K\"ahler form on $X$ is constructed in Lemma \ref{Kahler-metric}, the monotone Lagrangian torus of interest is the pre-image $L=\phi^{-1}(L_{\cl})$, and it lives in the complement of the ramification (anti-canonical) divisor $D$. In particular we have an area formula for discs $\beta\in H_2(X,L)$:
\begin{equation*}
 \langle \omega,\beta\rangle = \beta\cdot D.   
\end{equation*}

This formula explains in particular why we only need the power series ring $\mathbb{C}[[q]]$, as opposed to the Novikov ring $\Lambda_{\mathbb{C}}$.

Next, if we use the dimension formula (\ref{dimension-formula}), one sees that only discs of Maslov number $2$ contribute to the potential function $p$. As one expects, this potential function is tightly related to the Landau-Ginzburg potential $W$ that we computed in (\ref{potential-function}). The only difference is that when we defined $W$, we did not take areas into account, and as such we don't have the extra parameter $q$. Indeed, each bounding co-chain $b\in H^1(L,\mathbb{C})$ gives a local system $\xi_b$ on $L$:
\begin{align*}
    \xi_b: \pi_1(L)&\rightarrow \mathbb{C}^*\\
             \gamma &\mapsto \exp(\gamma\cdot b).
\end{align*}
Using the divisor axiom, it can be seen that:
\begin{equation}\label{potential-to-LG}
    p(b) = qW(L,\xi_b).
\end{equation}
In fact, when we compute the $b$-deformed $A_{\infty}$-algebra structure, it is the same as computing Fukaya's $A_{\infty}$-algebra structure for $(L,\xi_b)$.

Let $x,y\in H^1(L,\mathbb{C})\subseteq A^1$ be degree $1$ elements in our $A_{\infty}$-algebra $A$. Observe that we have:
\begin{equation}\label{m1-and-m2}
    m_1^b(x) = (dp)_b(x) \ \ \ \text{and} \ \ \ m_2^b(x,y) = (d^2p)_b(x,y).
\end{equation}
In fact, similar formulae hold for higher $A_{\infty}$-products as well. 
\begin{lemma}\label{critical is Clifford}
When $b$ is a non-degenerate critical point of $p$, $m_1^b = 0$, and we have an isomorphism of associative algebras:
$$(A,m^b_2) \cong \text{\emph{Cl}}(H^1(L,\mathbb{C}))\otimes \mathbb{C}[[q]].$$ 
\end{lemma}
\begin{proof}
The two equations in (\ref{m1-and-m2}) already give the desired result on $A^1$, and it suffices to show that $A$ is generated in degree $1$. 

Notice that if we drop the instanton corrections, the resulting $A_{\infty}$-algebra $A_0 = A\otimes \mathbb{C}[[q]]/(q)$ computes Fukaya's $A_{\infty}$-algebra of the exact Lagrangian manifold $L$ in the exact symplectic manifold $X\backslash D$, which is a formal exterior algebra on its degree $1$ part.

Now let $A_+ = \oplus_{i\geq 1} A^i$ be the ideal of all elements of positive degree, and consider the product map:
\begin{equation*}
    m_2^b : A^{\oplus 2}_+ \rightarrow A_+.
\end{equation*}
This is a map of finitely generated $\mathbb{C}[[q]]$-modules and, by our previous observation, it is surjective when restricted to the fiber at $0$; the unique maximal ideal of $\mathbb{C}[[q]]$. By Nakayama's lemma, we deduce that $m_2^b$ is surjective. Next, using the Leibniz rule, one sees that the differential $m_1^b$ vanishes identically, and that $A$ is generated in degree $1$. Therefore, the product structure is that of the usual Clifford algebra associated with $(d^2p)_b$. But recall that $b$ was assumed to be a non-degenerate critical point, so the lemma follows.
\end{proof}

Next, we need to compute the $A_{\infty}$-category associated with $L$ over $\mathbb{C}$. The underlying (now $\mathbb{Z}_2$-graded) vector space is:
\begin{equation*}
    \matheu{A} = H^*(L,\mathbb{C}),
\end{equation*}
and the $A_{\infty}$-structure maps $(\mu_k)_{k\geq 1}$ are the evaluations of $(m_k)_{k\geq 1}$ (from (\ref{structure maps})) at $q=1$.  There are no convergence issues to worry about because $L$ is monotone.

\begin{proposition}\label{A-side-HMS2}
The $\mathbb{Z}_2$-graded $A_{\infty}$-algebra $\matheu{A}$ is the formal Clifford algebra $\text{\emph{Cl}}_n(\mathbb{C})$ .
\end{proposition}
\begin{proof}
By combining the identity (\ref{potential-to-LG}), and the formula of $W$ from (\ref{potential-function}), we see that $b=0$ is a critical point of the the potential function $p$. Going back to the proof of Lemma \ref{critical is Clifford}, we have seen that $m_1^0 = 0$, and that $m_2^0$ is given by the Hessian of $p$ at $0$. By setting $q=1$, we get that $\mu^1=0$, and that $\mu^2$ follows the Hessian of a non degenerate function on $H^1(L,\mathbb{C})$. It follows that $H(\matheu{A})$ is the Clifford algebra $\text{\emph{Cl}}_n(\mathbb{C})$, which is known to be intrinsically formal: see for example \cite{sheridan-fano-mirror-symmetry}, Corollary 6.4.
\end{proof}

\begin{remark}
This method of computing Floer cohomology appears in the work of Sheridan (see \cite{sheridan-fano-mirror-symmetry}, Theorem 4.3) and also in the work of Fukaya-Ohta-Ono-Oh (see \cite{FOOO-toric-fibers}, Theorem 5.5), and before them in the work of Cho (see \cite{Cho-products}, theorem 5.6, also corollary 6.4). 
\end{remark}

\subsection{The B-side and HMS} The homological algebra of isolated hypersurface singularities is greatly studied in the work of Dyckerhoff \cite{Dyckerhoff}. It is shown there that $\text{D}^{\pi}_{\text{sg}}(W^{-1}(w_s))$ is generated by the skyscraper sheaf $\matheu{O}_p$ of the singular point. It is also shown in \cite{Dyckerhoff} (see also \cite{Orlov-formal-completion}) that this category only depends on the formal completion of a neighborhood of the singular point. In particular, we have an equivalence of triangulated categories:
\begin{equation*}
    \D^{\pi}_{\text{sg}}(W^{-1}(w_s)) = D^{\pi} \text{MF}(\mathbb{C}[[z_1,\dots,z_n]],z_1^2+\dots+z_n^2).
\end{equation*}

In the equivalence above, passing to matrix factorizations requires a stabilization procedure explained in section 2 of \cite{Dyckerhoff}. The category $\text{MF}$ of matrix factorizations is a $\mathbb{Z}_2$-graded category and in this case it is generated by the (\emph{stablization} of the) residue field $\mathbb{C}$.  In section 5.5 of that same paper, the self-hom space is computed to be $\text{Cl}_n(\mathbb{C})$ with an identically vanishing differential.

Combining all of this together, we get:
\begin{lemma}\label{B-side-HMS2}
There is an equivalence of triangulated categories:
\begin{equation*}
  \emph{\D}^{\pi}_{\text{sg}}(W^{-1}(w_s)) = \emph{D}^{\pi}\text{Cl}_n(\mathbb{C}).  
\end{equation*}
\end{lemma}

\begin{remark}
We refer the reader to the work of J. Smith in \cite{JackSmith}, for a recent treatment of the homological algebra of isolated hypersurface singularities that is more adapted to homological mirror symmetry.
\end{remark}

We have now collected all the necessary ingredients to prove the main result of this section.

\textit{proof of Theorem \ref{HMS2}} Recall that the eigenspace corresponding to $w_s$ in: 
\begin{equation*}
    c_1\star(-): \text{QH}(X) \rightarrow \text{QH}(X),
\end{equation*}
has dimension 1, and as a consequence, any object in $\text{Fuk}(X)_{w_s}$ with non-zero Floer cohomology will split-generate. Refer to Corollary 2.19 and Proposition 7.11 of \cite{sheridan-fano-mirror-symmetry} for more details. In particular, the monotone Lagrangian torus $L$ split-generates. We have already computed its associated Fukaya $A_{\infty}$-algebra in Proposition \ref{A-side-HMS2}. Combining that with the result of Lemma \ref{B-side-HMS2}, we deduce the desired equivalence:
\begin{equation*}
    \D^{\pi}\text{Fuk}(X)_{w_s} \cong \D^{\pi}_{\text{sg}}(W^{-1}(w_s)).
\end{equation*}
\qed

\section{HMS in the toric limit}
In Proposition \ref{formula-for-W}, we counted Maslov index $2$ discs with boundary on the monotone Lagrangian torus $L\subseteq X_f$ constructed in Proposition \ref{partial-fibration}, where $X_f$ is the smooth index 1 Fano hypersurface in projective space $\mathbb{P}^{n+1}$, cut-out by an equation of the form:
\begin{equation*}
    X_f=V(t^{n+1}-f(x_0,\dots,x_n)),
\end{equation*}
where $f$ is a homogeneous polynomial of degree $n+1$, that is sufficiently close to the toric boundary $f_0=x_0\dotsi x_n$. The limit of these index $1$ Fano hypersurfaces is the singular toric Fano variety:
\begin{equation*}
    X_0 = V(t^{n+1}-x_0\dots x_n).
\end{equation*}
The super-potential function associated with $L$ has the following formula:
\begin{equation}\label{W_L}
    W_L = \frac{(1+y_1+\dots+y_n)^{n+1}}{y_1\dots y_n}-(n+1)!.
\end{equation}
In the mirror symmetry literature, the pair $(X_0,W)$ is a called a toric Landau-Ginzburg model for the index $1$ Fano hypersurface, we refer the reader to \cite{Katzarkov-Przyj} for more context.

Our goal for this section is to study homological mirror symmetry for $X_0$, which we view as the $B$-side, and its mirror super-potential $W_L$, which we view as the $A$-side. While the translation term $(n+1)!$ is crucial in the full HMS story of index $1$ Fano hypersurfaces, it actually has no bearing on the particular version of HMS we consider in the present section. Because of that, we simply drop the translation term and work with the following instead:
\begin{equation}\label{W}
       W = \frac{(1+y_1+\dots+y_n)^{n+1}}{y_1\dots y_n}.
\end{equation}

We associate with $W$ a Fukaya-Seidel $A_{\infty}$-category $\FS((\mathbb{C}^*)^n,W)$ using the Lagrangian thimbles of $W$. We explain how this Fukaya-Seidel category recovers the homogeneous coordinate ring of $X_0$. More precisely, we prove the following result:
\begin{theorem}\label{hypersurface is B}
There is a collection of Lefschetz thimbles $L_i$ in $((\mathbb{C}^*)^n,W)$ such that:
\begin{equation*}
    HW(L_i,L_j) \simeq \hom(\matheu{O}_{X_0}(i),\matheu{O}_{X_0}(j)).
\end{equation*}
Furthermore, the isomorphisms above are compatible with the relevant product structures. 
\end{theorem}
The main insight we use is a base-cover relationship between $((\mathbb{C}^*)^n,W)$ and $((\mathbb{C}^*)^n,W_{\cl})$, together with a folklore result on homological mirror symmetry for projective space $\mathbb{P}^n$. Indeed, recall that the Landau-Ginzburg model associated with projective space is $((\mathbb{C}^*)^n,W_\cl)$, where:
\begin{equation}\label{formula-for-Wcl}
    W_{\cl} = y_1+\dots+y_n+\frac{1}{y_1\dotsi y_n}.
\end{equation}
There is a free action of $\mathbb{Z}_{n+1}$ on $(\mathbb{C}^*)^n$ that rotates the coordinates by $(n+1)^{\text{th}}$-roots of unity:
\begin{equation*}
 \zeta\cdot(y_1,\dots,y_n)\mapsto (\zeta\cdot y_1\dots,\zeta\cdot y_n).   
\end{equation*}
The potential function $W_{\cl}$ is not $\mathbb{Z}_{n+1}$-invariant, but its power $W_{\cl}^{n+1}$ is, and in fact:
\begin{equation}\label{unbranched-covering}
    ((\mathbb{C}^*)^n,W) = ((\mathbb{C}^*)^n/\mathbb{Z}_{n+1},W^{n+1}_{\cl}).
\end{equation}
The quotient map is:
\begin{align}\label{covering map formula}
 \pi:(\mathbb{C}^*)^n &\rightarrow (\mathbb{C}^*)^n\\
 (y_1,\dots,y_n) &\mapsto (y_1Y,\dots,y_nY),\notag
\end{align}
where $Y=y_1\dotsi y_n$. The unbranched covering map $\pi$ seems to mirror the branched covering map $\phi:X_0\rightarrow\mathbb{P}^n$. This mirror correspondence looks like:
\begin{align*}
    \pi^{-1}(-) &\longleftrightarrow \phi_*(-)\\
    \pi(-) &\longleftrightarrow \phi^*(-).
\end{align*}
Our approach to proving Theorem \ref{hypersurface is B} is guided by this correspondence: the methods we use suggest that there exists a commutative diagram of triangulated categories:
\begin{equation*}
    \begin{tikzcd}
    D^{\pi}\FS((\mathbb{C}^*)^n,W) \arrow[r,"\text{HMS}_{X_0}"]\arrow[d,"\psi"]&
    \Perf(X_0)\arrow[d,"\phi"]&\\
    D^{\pi}\FS((\mathbb{C}^*)^n,W_\cl)\arrow[r,"\text{HMS}_{\mathbb{P}^n}"]& D^b\Coh(\mathbb{P}^n), &
    \end{tikzcd}
\end{equation*}
such that both horizontal arrows are equivalences.

\subsection{Partially wrapped Floer theory} We fix a base field $k=\mathbb{C}$.
The Fukaya-Seidel $A_{\infty}$-category associated with the Landau-Ginzburg model $((\mathbb{C}^*)^n,W)$ is constructed by counting holomorphic polygons with boundary on (wrappings of) a collection of Lagrangians. The role of $W$ is to \emph{stop} (in the sense of Z.Sylvan \cite{Sylvan}) the wrapping at a regular fiber of $W$.  When $W$ has only non-degenerate singularities, this is exactly the Fukaya-Seidel category defined for example in \cite{PL-theory}. Because in our case, one of the two singularities of $W$ is non-degenerate, we instead resort to the more recent work of Ganatra-Pardon-Shende in \cite{GPS1} and \cite{GPS2}, although our set-up is actually closer to \cite{Abouzaid2006}. The Liouville structure on $(\mathbb{C}^*)^n$ comes from the 1-form:
\begin{equation*}
    \theta = \sum r_id\theta_i,
\end{equation*}
where $(r_i,\theta_i)$ are the radial and angular components of $i^{\text{th}}$-coordinate $y_i\in\mathbb{C}^*$. It can also be seen as the Stein structure coming from the pluri-subharmonic function:
\begin{equation}\label{psh function}
    h = \abs{y_1} + \dotsi +\abs{y_n}.
\end{equation}
Let $R$ be a fixed, large enough positive number. The objects of $FS((\mathbb{C}^*)^n,W)$ are Lefschetz thimbles $L_{\gamma}$ corresponding to embedded paths $\gamma:[0,1]\rightarrow \mathbb{C}$ such that:
\begin{itemize}
    \item[-] $\abs{\gamma(1)}=R$ but $\gamma(1)\neq -R$.
    \item[-] $\gamma(0)$ is a non-degenerate critical value of $W$.
\end{itemize}
The first condition means that we will stop our wrapped Floer theory at the Weinstein hypersurface $W^{-1}(-R)$. Such $\gamma$ is sometimes called a \emph{vanishing path}.

Because we are only restricting to Lefschetz thimbles, we note that this category (even after taking triangulated split-closures) is a-priori \emph{smaller} than the stopped category $\text{WF}((\mathbb{C}^*)^n, W^{-1}(R))$ in the language of \cite{GPS1}. For example, when $W$ is the Laurent polynomial in (\ref{formula-for-Wcl}), thimbles are enough to recover the full stopped category. However, when $W$ is the Laurent polynomial from (\ref{W_L}), they are not. 

Let $L_1$ and $L_2$ be two objects in $\FS((\mathbb{C}^*)^n,W)$. The holomorphic convexity of $(\mathbb{C}^*)^n$, together with exactness of the Lagrangians $L_i$, ensure that we have the necessary compactness to define a Floer cohomology vector space $HF(L_1,L_2)$ over $k$. However, these vector spaces fails to be independent of Hamiltonian isotopies. Indeed, as $L_1$ is wrapped positively to $L_1^+$ (or $L_2$ wrapped negatively to $L_2^{-}$), the pair $(L_1^+,L_2)$ will likely acquire more intersection points and the vector space $HF(L_1^+,L_2)$ "grows" bigger as a consequence. More accurately, there is a continuation map:
\begin{equation*}\label{continuation}
c: HF(L_1,L_2) \rightarrow  HF(L_1^+,L_2).   
\end{equation*}
One therefore defines (see \cite{GPS1}) a wrapped Floer cohomology group by the following recipe:
\begin{equation}\label{wrapped floer cohomology-1}
    HW(L_1,L_2) = \varinjlim_{w} HF(L^w_1,L_2),
\end{equation}
where the limit is taken over all positive wrappings $L^w_1$ that do not cross the stop $W^{-1}(-R)$. This is now invariant under Hamiltonian isotopies, up to canonical isomorphism.

In the case of a pair $(L_{\gamma_1},L_{\gamma_2})$ of Lefschetz thimbles, this recipe simplifies: we can get positive wrappings of $L_{\gamma_1}$ by instead wrapping the underlying vanishing path $\gamma_1$ around the boundary of the disc of radius $R$. Notice however that once we wrap $\gamma_1$ to a path $\gamma^+_1$ whose end-point $\gamma^+_1(1)$ is closer to the stop $-R$ (in the anti-clockwise direction) than $\gamma_2(1)$, we no longer gain any new intersection points by positively wrapping $\gamma$ even further. As a consequence:
\begin{equation}\label{wrapped floer cohomology-2}
    HW(L_{\gamma_1},L_{\gamma_2}) = HF(L_{\gamma^+_1},L_{\gamma_2}).
\end{equation}
This is basically how stopped Floer cohomology was defined for Fukaya-Seidel categories before Z. Sylvan introduced stops in \cite{Sylvan}. See for example \cite{Abouzaid2006} section 2, or \cite{PL-theory} chapter 3. These vector spaces can be upraded into an $A_{\infty}$-category by counting holomorphic polygons:
\begin{equation*}
   \mu^d: CF(L_{\gamma_{d-1}},L_{\gamma_d})\otimes\dotsi\otimes CF(L_{\gamma_0},L_{\gamma_1})\rightarrow CF(L_{\gamma_0},L_{\gamma_d})[2-d],
\end{equation*}
whenever the sequence of boundary points $\gamma_0(1),\gamma_1(1),\dots,\gamma_d(1)$ is ordered clock-wise in the arc $\{\abs{z}=R\}\backslash \{-R\}$. Finally, because the Lagrangians $L_{\gamma}$ are contractible, they carry canonical spin structures to orient the moduli spaces of holomorphic polygons, and grading data to make $\FS((\mathbb{C}^*)^n,W)$ a $k$-linear, $\mathbb{Z}$-graded $A_{\infty}$-category.

In the previous construction, we may stop the wrapping in Floer cohomology even further by adding more stops of the form $W^{-1}(z)$, where $z$ spans a finite subset $I$ of the circle $\{\abs{z}=R\}$. This means that in equations (\ref{wrapped floer cohomology-1}) and (\ref{wrapped floer cohomology-2}), the positive wrappings stop before running into either one of the fibers in $W^{-1}(I)$. We denote the resulting $A_{\infty}$-category by $\FS((\mathbb{C}^*)^n,W,I)$. For example:
\begin{equation*}
\FS((\mathbb{C}^*)^n,W,-R)   = \FS((\mathbb{C}^*)^n,W). 
\end{equation*}
Given two finite collections of stops $I\subseteq J \subseteq \{\abs{z}=R\}$, the extra wrapping the may occur in $\FS((\mathbb{C}^*)^n,W,I)$, produces continuation elements:
\begin{equation*}
    c_{I\subseteq J}: HW_{J}(L_1,L_2) \rightarrow HW_{I}(L_1,L_2).
\end{equation*}
These continuation elements can in fact be upgraded to an $A_{\infty}$-functor:
\begin{equation*}
    c:\FS((\mathbb{C}^*)^n,W,J)\rightarrow \FS((\mathbb{C}^*)^n,W,I),
\end{equation*}
which is sometimes called \emph{stop-removal}. This functor is carefully constructed in \cite{GPS1} and thoroughly studied in \cite{GPS2}.

\begin{remark}
In our presentation here, we work as though $k$ is a field of characteristic $2$, so as to avoid cluttering the main ideas with notation. In reality, intersection points of Lagrangians should be interpreted as trivializations of orientation operators coming from the Fredholm theory of the $\overline{\partial}$-equation. We refer the reader to \cite{PL-theory}, section 11 for the exact details on how this works.
\end{remark}

\subsection{The A-side, unbranched coverings} We now restrict our discussion of Fukaya categories to the context of the base-cover relationship in (\ref{unbranched-covering}). The potential function $W$ from (\ref{W}) has one non-degenerate critical value at $w_s=(n+1)^{n+1}$, and then a big critical value $w_b=0$. Therefore, the Lefschetz thimbles $L_{\gamma}$ in $\FS((\mathbb{C}^*)^n,W)$ are classified by their monodromy around $0$, which also can be thought of as the intersection number of $\gamma$ with the segment $(-R,0)$.
\begin{definition}\label{thimbles in the base}
For an integer $i\in\mathbb{Z}$, the Lagrangian $L_i\in ((\mathbb{C}^*)^n,W)$ is the Lefschetz thimble associated with an embedded path $\gamma:[0,1]\rightarrow \mathbb{C}\backslash\{0\}$, such that $\gamma(1)=R$,  $\gamma(0)$ is the non-degenerate critical value $w_s=(n+1)^{n+1}$, and the path's clockwise winding number around $0$, relative to the endpoints $w_s$ and $R$, is $i$. 
\end{definition}
\vspace{-1cm}
\hspace{-.8cm}{
\begin{tikzpicture}[x=0.75pt,y=0.75pt,yscale=-.8,xscale=.8]

\draw   (99,138.5) .. controls (99,66.43) and (197.27,8) .. (318.5,8) .. controls (439.73,8) and (538,66.43) .. (538,138.5) .. controls (538,210.57) and (439.73,269) .. (318.5,269) .. controls (197.27,269) and (99,210.57) .. (99,138.5) -- cycle ;
\draw  [fill={rgb, 255:red, 208; green, 2; blue, 27 }  ,fill opacity=1 ] (95.5,149.5) .. controls (95.5,147.01) and (97.51,145) .. (100,145) .. controls (102.49,145) and (104.5,147.01) .. (104.5,149.5) .. controls (104.5,151.99) and (102.49,154) .. (100,154) .. controls (97.51,154) and (95.5,151.99) .. (95.5,149.5) -- cycle ;
\draw  [fill={rgb, 255:red, 144; green, 19; blue, 254 }  ,fill opacity=1 ] (239.5,145) -- (241.7,147.98) -- (246.63,148.45) -- (243.07,150.77) -- (243.91,154.05) -- (239.5,152.5) -- (235.09,154.05) -- (235.93,150.77) -- (232.37,148.45) -- (237.3,147.98) -- cycle ;
\draw  [fill={rgb, 255:red, 126; green, 211; blue, 33 }  ,fill opacity=1 ] (365.5,146) -- (372,155) -- (359,155) -- cycle ;
\draw    (365.5,150.5) -- (539,149.5) ;
\draw    (365.5,150.5) .. controls (18,-70) and (117,433) .. (539,149.5) ;

\draw (434,127) node [anchor=north west][inner sep=0.75pt]   [align=left] {$L_0$};
\draw (385,225) node [anchor=north west][inner sep=0.75pt]   [align=left] {$L_1$};
\draw (229,155) node [anchor=north west][inner sep=0.75pt]   [align=left] {$w_b$};
\draw (353,160) node [anchor=north west][inner sep=0.75pt]   [align=left] {$w_s$};
\draw (60,140) node [anchor=north west][inner sep=0.75pt]   [align=left] {$-R$};
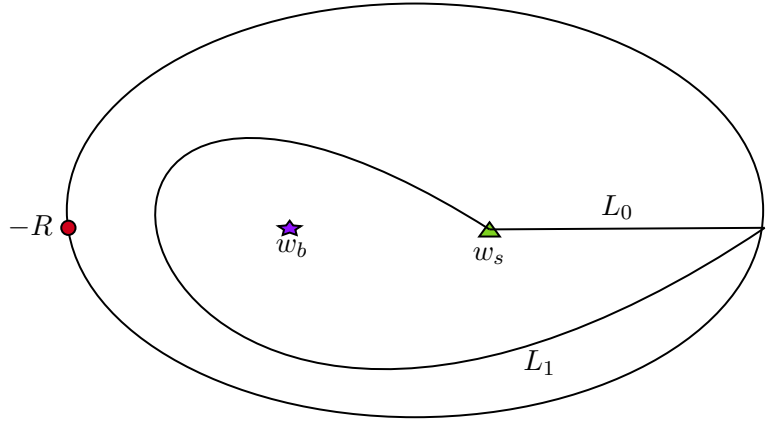
\draw (215,283) node [anchor=north west][inner sep=0.75pt]   [align=left] {Some Lefschetz thimbles for $W$.};
\end{tikzpicture}
}\vspace{-2.5cm}
\\
The unbranched covering map $\pi$ from (\ref{covering map formula}) induces an $A_{\infty}$-functor:
\begin{equation*}
    \pi:\FS((\mathbb{C}^*)^n,W,-R))\rightarrow \FS((\mathbb{C}^*)^n,W_\cl,J)),
\end{equation*}
where the collection $J$ of stops is:
\begin{equation*}
   J = \{z\in\mathbb{C}\ |\  z^{n+1} = -R\}.
\end{equation*}
At the level of objects, this functor maps a Lagrangian thimble to its pre-image. At the level of hom spaces, the chain map:
\begin{equation}\label{pre-image-before-continuation}
    \pi^1:CW(L_i,L_j) \rightarrow CW_J(\pi^{-1}L_i,\pi^{-1}L_j),
\end{equation}
takes an intersection point $p\in L_i\cap L_j$ to the sum of its pre-images. As an $A_{\infty}$-functor, the higher components all vanish, i.e $\pi^d=0$ for all $d\geq 2$. The reason that $\pi^1$ above is a chain map (and in fact respects the $A_{\infty}$-structures) is because the pre-images $\pi^{-1}(L_i)$ have $n+1$ connected components lying in different sheets of the covering map, one for each critical value of $W_{\cl}$. By the homotopy lifting property, a holomorphic strip with boundary on $(L_0,L_1)$ has exactly $n+1$-lifts via $\pi$, which again lie each in a different sheet of the covering map.

\begin{remark}
A few observations regarding the previous definition are in order:
\begin{itemize}
    \item[-] For the picture above to work perfectly, we need to choose the pluri-subharmonic function on the bottom $(\mathbb{C}^*)^n$ to be the descent of $h$ (as in (\ref{psh function})) through the covering map. 
    \item[-] In the map (\ref{pre-image-before-continuation}), the point $p$ should be replaced by its orientation line $o(p)$. The pre-images $\pi^{-1}(L_i)$  and $\pi^{-1}(L_j)$ inherit their brane structures from those of $L_i$ and $L_j$. Because $\pi$ is unbranched, for each intersection point $q\in \pi^{-1}(L_i)\cap\pi^{-1}(L_j)$, there is a canonical isomorphism of orientation lines $o(q)\simeq o(\pi(q))$, this is what should be used to define $\pi^1$.
\end{itemize}
\end{remark}
Next, we push our Lagrangian thimbles to $((\mathbb{C}^*)^n,W_{\cl})$ using the acceleration functor:
\begin{equation*}
c:\FS((\mathbb{C}^*)^n,W_{\cl},J)\rightarrow \FS((\mathbb{C}^*)^n,W_{\cl},s_1),
\end{equation*}
where the stop $s_1$ is the one located immediately after $\sqrt[n+1]{R}$ in the counter-clockwise direction:
\begin{equation*}
    s_1 = R^{\frac{1}{n+1}}e^{\frac{\pi i}{n+1}}.
\end{equation*}
Finally we define the $A_{\infty}$-functor $\psi$ as the composition of $\pi$ and $c$:
\begin{equation}\label{psi functor}
    \psi: \FS((\mathbb{C}^*)^n,W))\rightarrow \FS((\mathbb{C}^*)^n,W_\cl,s_1).
\end{equation}
We refer the reader to the figure below for some intuition. It turns out that the functor $\psi$ mirrors the pushforward map $\phi_*$ on perfect complexes.

The Landau-Ginzburg model $((\mathbb{C}^*)^n,W_\cl)$ has been extensively studied in the literature as the mirror to projective space. P.Seidel studied the case $n=2$ in \cite{Seidel-vanishing-cycles-and-mutations}, section 3. M. Abouzaid then proved HMS for all smooth toric Fano varieties in \cite{Abouzaid-toric-HMS}, and a quick summary of that story in the case of $\mathbb{P}^n$ can be found in D.Auroux's speculations \cite{Auroux-speculations}, section 7. We will rely on the more recent treatment of Futaki-Ueda in \cite{Futaki-Ueda}. We now briefly recall the elements of that story that are most pertinent to our work.

Following the set-up of the previous discussion, we consider Lagrangian thimbles $\hat{L}_{\gamma}$ whose underlying vanishing path is an embedding:
\begin{equation*}
\gamma:[0,1]\rightarrow \{n+1\leq \abs{z}\leq R^{\frac{1}{n+1}}\},   
\end{equation*}
satisfying the following properties:
\begin{itemize}
    \item[-] $\abs{\gamma(1)} = R$ and $\gamma(1)\neq s_1$.
    \item[-] $\gamma(0)$ is one of the $n+1$ critical values of $W_{\cl}$.
\end{itemize}

These vanishing paths depend on $2$ pieces of data. The first is the choice of a critical value:
\begin{equation*}
    \gamma(0)=w\in \{n+1,(n+1)\zeta,\dots, (n+1)\zeta^n\}.
\end{equation*}
After $\gamma(0)=w$ has been fixed, $\gamma$ only depends on the amount of winding it does with respect to the stop. To quantify this amount, we fix $\gamma_{w,0}$ to be the radial path from $w$ to the circle $\{\abs{z}=R^{\frac{1}{n+1}}\}$. Then $\gamma_{w,i}$ will be obtained from $\gamma_{w,0}$ by further winding  the endpoint $\gamma_{w,0}(1)$ in the clockwise direction until it crosses the stop $s_1$, $i$ times.
\begin{definition}
Given a critical value $w=(n+1)\zeta^{-k}$ of $W_{\cl}$ and an integer $i\in\mathbb{Z}$, the Lagrangian $\hat{L}_{k,i}$ is the Lefschetz thimble associated with the path $\gamma_{w,i}$ as described above. See figure below for examples.
\end{definition}

\begin{center}
{\tikzset{every picture/.style={line width=0.75pt}} 
\begin{tikzpicture}[x=0.75pt,y=0.75pt,yscale=-.7,xscale=.7]

\draw   (63,150.5) .. controls (63,75.11) and (182.99,14) .. (331,14) .. controls (479.01,14) and (599,75.11) .. (599,150.5) .. controls (599,225.89) and (479.01,287) .. (331,287) .. controls (182.99,287) and (63,225.89) .. (63,150.5) -- cycle ;
\draw    (374.75,150.5) -- (599,150.5) ;
\draw    (331,14) -- (331,106.75) ;
\draw    (63,150.5) -- (287.25,150.5) ;
\draw    (331,194.25) -- (331,287) ;
\draw    (599,150.5) .. controls (571.21,163.81) and (546.08,159.95) .. (522.86,148.26) .. controls (447.48,110.34) and (392.16,-10.03) .. (331,106.75) ;
\draw    (331,14) .. controls (428.7,69.71) and (174.98,22.42) .. (206.98,84.52) .. controls (214.61,99.34) and (238.49,120.37) .. (287.25,150.5) ;
\draw    (63,150.5) .. controls (215.15,96.61) and (292.66,290.02) .. (323.02,223.93) .. controls (326.18,217.04) and (328.83,207.34) .. (331,194.25) ;
\draw    (331,287) .. controls (265,229) and (507,192) .. (374.75,150.5) ;
\draw  [fill={rgb, 255:red, 247; green, 22; blue, 22 }  ,fill opacity=1 ] (494,45) .. controls (494,41.69) and (496.69,39) .. (500,39) .. controls (503.31,39) and (506,41.69) .. (506,45) .. controls (506,48.31) and (503.31,51) .. (500,51) .. controls (496.69,51) and (494,48.31) .. (494,45) -- cycle ;
\draw    (89,46) -- (195.05,70.55) ;
\draw [shift={(197,71)}, rotate = 193.03] [color={rgb, 255:red, 0; green, 0; blue, 0 }  ][line width=0.75]    (10.93,-3.29) .. controls (6.95,-1.4) and (3.31,-0.3) .. (0,0) .. controls (3.31,0.3) and (6.95,1.4) .. (10.93,3.29)   ;
\draw    (82,62) -- (121.05,134.24) ;
\draw [shift={(122,136)}, rotate = 241.61] [color={rgb, 255:red, 0; green, 0; blue, 0 }  ][line width=0.75]    (10.93,-3.29) .. controls (6.95,-1.4) and (3.31,-0.3) .. (0,0) .. controls (3.31,0.3) and (6.95,1.4) .. (10.93,3.29)   ;
\draw    (87,55) -- (398.16,189.21) ;
\draw [shift={(400,190)}, rotate = 203.32999999999998] [color={rgb, 255:red, 0; green, 0; blue, 0 }  ][line width=0.75]    (10.93,-3.29) .. controls (6.95,-1.4) and (3.31,-0.3) .. (0,0) .. controls (3.31,0.3) and (6.95,1.4) .. (10.93,3.29)   ;
\draw    (91,38) -- (431.02,85.72) ;
\draw [shift={(433,86)}, rotate = 187.99] [color={rgb, 255:red, 0; green, 0; blue, 0 }  ][line width=0.75]    (10.93,-3.29) .. controls (6.95,-1.4) and (3.31,-0.3) .. (0,0) .. controls (3.31,0.3) and (6.95,1.4) .. (10.93,3.29)   ;

\draw (502,26.4) node [anchor=north west][inner sep=0.75pt]    {$s_{1}$};
\draw (312,114.4) node [anchor=north west][inner sep=0.75pt]    {$\zeta $};
\draw (276,159.4) node [anchor=north west][inner sep=0.75pt]    {$\zeta ^{2}$};
\draw (335,195.4) node [anchor=north west][inner sep=0.75pt]    {$\zeta ^{3}$};
\draw (450,155.4) node [anchor=north west][inner sep=0.75pt]    {$\hat{L}_{0,0}$};
\draw (329,255.4) node [anchor=north west][inner sep=0.75pt]    {$\hat{L}_{1,0}$};
\draw (105,158.4) node [anchor=north west][inner sep=0.75pt]    {$\hat{L}_{2,0}$};
\draw (290,78.4) node [anchor=north west][inner sep=0.75pt]    {$\hat{L}_{3,0}$};
\draw (22,35.4) node [anchor=north west][inner sep=0.75pt]    {$\psi ( L_{1})$};
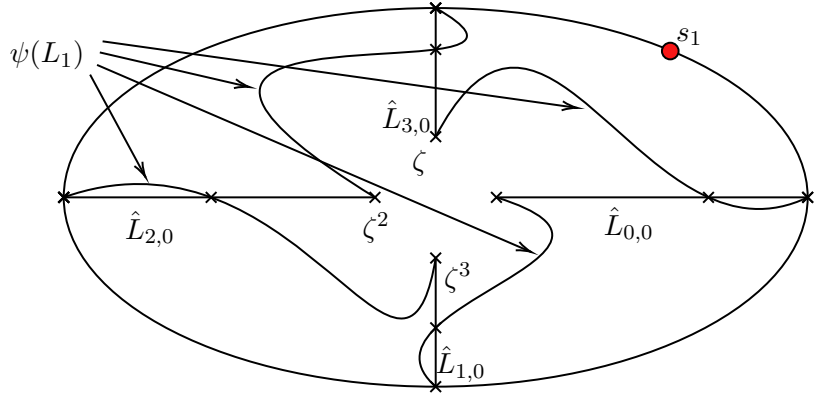
\draw (90,305) node [anchor=north west][inner sep=0.75pt]   [align=left] {The action of $\displaystyle \psi $ on $\displaystyle L_{0}$ and $\displaystyle L_{1}$. This figure is for $\displaystyle W_{\text{cl}} .$};

\draw   (327.15,283.15) -- (334.85,290.85)(334.85,283.15) -- (327.15,290.85) ;
\draw   (327.15,190.4) -- (334.85,198.1)(334.85,190.4) -- (327.15,198.1) ;
\draw   (327.15,240.72) -- (334.85,248.41)(334.85,240.72) -- (327.15,248.41) ;
\draw   (595.15,146.65) -- (602.85,154.35)(602.85,146.65) -- (595.15,154.35) ;
\draw   (595.15,146.65) -- (602.85,154.35)(602.85,146.65) -- (595.15,154.35) ;
\draw   (523.7,146.65) -- (531.39,154.35)(531.39,146.65) -- (523.7,154.35) ;
\draw   (370.9,146.65) -- (378.6,154.35)(378.6,146.65) -- (370.9,154.35) ;
\draw   (327.15,10.15) -- (334.85,17.85)(334.85,10.15) -- (327.15,17.85) ;
\draw   (327.15,102.9) -- (334.85,110.6)(334.85,102.9) -- (327.15,110.6) ;
\draw   (327.15,10.15) -- (334.85,17.85)(334.85,10.15) -- (327.15,17.85) ;
\draw   (327.15,40.01) -- (334.85,47.7)(334.85,40.01) -- (327.15,47.7) ;
\draw   (59.15,146.65) -- (66.85,154.35)(66.85,146.65) -- (59.15,154.35) ;
\draw   (283.4,146.65) -- (291.1,154.35)(291.1,146.65) -- (283.4,154.35) ;
\draw   (59.15,146.65) -- (66.85,154.35)(66.85,146.65) -- (59.15,154.35) ;
\draw   (165.37,146.65) -- (173.06,154.35)(173.06,146.65) -- (165.37,154.35) ;
\draw   (327.15,10.15) -- (334.85,17.85)(334.85,10.15) -- (327.15,17.85) ;
\draw   (59.15,146.65) -- (66.85,154.35)(66.85,146.65) -- (59.15,154.35) ;
\draw   (59.15,146.65) -- (66.85,154.35)(66.85,146.65) -- (59.15,154.35) ;
\end{tikzpicture}}
\end{center}

We now state a folklore result in homological mirror symmetry. It will facilitate the comparison between the A-side calculations we do next, with their B-side counterparts. We provide a more detailed discussion of this equivalence in section \ref{HMS-for-cpn}.
\begin{theorem}(see \cite{Futaki-Ueda}, \cite{Abouzaid-toric-HMS})\label{HMS for projective space}
There is an $A_{\infty}$-functor:
\begin{equation*}
    \theta: \FS((\mathbb{C}^*)^n,W_\cl)\rightarrow \Coh_{\dg}(\mathbb{P}^n),
\end{equation*}
that induces a quasi-equivalence of split-closed triangulated categories:
\begin{equation}\label{theta-functor}
\theta: D^{\pi}\FS((\mathbb{C}^*)^n,W_\cl)\rightarrow D^b(\Coh(\mathbb{P}^n)).    
\end{equation}
At the level of objects, this functor maps $\hat{L}_{k,i}$ to $\matheu{O}_{\mathbb{P}^n}(-k+i(n+1))$. 
\end{theorem}

We now go back to the $A_{\infty}$-functor $\psi$ defined in (\ref{psi functor}). We start by computing its action on objects. 
\begin{lemma}
Let $j\in\mathbb{Z}$ be an integer given in the form $j=q(n+1)+r$ with $0\leq r\leq n$, and let $L_j$ be the exact Lagrangian from Definition \ref{thimbles in the base}. Then:
\begin{equation}\label{direct sum decomp of pre-image}
    \psi(L_j) = \bigoplus_{k=0}^n L_{k,j_k},
\end{equation}
where:
\begin{equation*}
    j_k=
    \begin{cases} 
      q & \ \text{if}\ 0\leq k\leq n-r, \\
      q+1 &\ \text{if}\  k>n-r.
   \end{cases}
\end{equation*}
\end{lemma}

\proof We assume $j\geq 0$ in order to simplify the phrasing of the argument. The Lagrangians $\hat{L}_{k,j_k}$ are the connected components of $\psi(L_j)$, so the direct sum decomposition is automatic. The only work that needs be done is in identifying the winding numbers $j_k$. In the base $((\mathbb{C}^*)^n,W)$, the wrapping $L_0\rightsquigarrow L_j$ follows the angles $\exp(-2\pi i t)$, with $0\leq t \leq j$. When this wrapping is lifted to $\hat{L}_{k,0}\rightsquigarrow \hat{L}_{k,j_k}$, it follows the angles:
\begin{equation*}
    \theta_t = \exp\left(\frac{2\pi i}{n+1}(n+1-k-t)\right).
\end{equation*}
The integer $j_k$ is now simply the number of times this path of angles crosses the stop $s_1=\exp\left(\frac{\pi i}{n+1}\right)$. This is the same as counting the number of elements in the set:
\begin{equation*}
    \left\{ t\in [0,j] \ |\ t+k+\frac{1}{2} \equiv 0 \mod (n+1)\mathbb{Z}\right\}.
\end{equation*}
Using the Euclidean division $j=q(n+1)+r$, we see that this number is $q$, plus however many multiples of $n+1$ are in the interval:
\begin{equation*}
    \left[k+\frac{1}{2},k+r+\frac{1}{2}\right].
\end{equation*}
Because $k,r<n$, this interval either contains $1$ such multiple (if $k+r>n$) or none at all (if $k+r\leq n$). The formula for $j_k$ then follows. \qed
\begin{remark}\label{pre-image is push forward}
In light of the homological mirror symmetry statement in Theorem \ref{HMS for projective space}, it is worth noting that the numbers $j_k$ in the previous Lemma work out perfectly so that:
\begin{equation*}
    \theta(\psi(L_j)) = \bigoplus_{k=0}^n \matheu{O}_{\mathbb{P}^n}(j-k).
\end{equation*}
\end{remark}
In the direct sum decomposition (\ref{direct sum decomp of pre-image}) above, the direct summand with index $k_+ = n+1-r$ is "more positive" than all the others. The next lemma makes this idea more precise.

\begin{lemma}\label{A side isomorphism}
In the context of the previous lemma, let $p\in \mathbb{Z}$ be another integer. Then the composition:
\begin{equation*}
    HW(L_j,L_p)\rightarrow HW(\psi(L_j),\psi(L_p))\rightarrow HW(\hat{L}_{k_+,j_{k_+}},\psi(L_p)).
\end{equation*}
is an isomorphism.
\end{lemma}

\proof Observe that the composition: 
\begin{equation}\label{iso with stops}
    HW(L_j,L_p)\rightarrow HW(\pi^{-1}(L_j),\pi^{-1}(L_p))\rightarrow HW(\hat{L}_{k,j_{k}},\pi^{-1}(L_p)),
\end{equation}
is an isomorphism for all $k=0,\dots,n$, because the intersection points in $CF(L_j,L_p)$ are in 1-to-1 correspondence with those of $CF(\hat{L}_{k,j_{k}},\pi^{-1}(L_p))$, and the pair $(\hat{L}_{k,j_{k}},\pi^{-1}(L_p))$ acquires no further wrapping in the category $\FS((\mathbb{C}^*)^n,W_\cl,J))$. When we remove all the stops but $s_1$, many of the pairs $(\hat{L}_{k,j_{k}},\pi^{-1}(L_p))$ will acquire more wrapping. This phenomenon can be studied by examining the angle where $\hat{L}_{k,j_k}$ hits the boundary. This angle is:
\begin{equation*}
    \frac{-2\pi}{n+1}(k+j).
\end{equation*}
Recall that the stop $s_1$ sits at an angle of $\pi/(n+1)$. In particular, when $k=k_+$, the boundary of $\hat{L}_{k,j_k}$ is as close to the stop as any $\hat{L}_{k,p_k}$ can be. In particular, the pair $(\hat{L}_{k_+,j_{k_+}},\psi(L_p))$ is sufficiently wrapped, and the Lemma now follows from (\ref{iso with stops}). \qed

\begin{remark}
In light of the homological mirror symmetry statement in Theorem \ref{HMS for projective space}, the previous Lemma mirrors the adjunction isomorphism:
\begin{equation*}
 \hom_{X_0}(\matheu{O}_{X_0}(i),\matheu{O}_{X_0}(j))\rightarrow  \hom(\matheu{O}_{\mathbb{P}^n}(i),  \phi_*\matheu{O}_{X_0}(j))  .
\end{equation*}
\end{remark}
The previous Lemma computes $ HW(L_j,L_p)$ as a quotient (as opposed to a subspace) of $HW(\psi(L_j),\psi(L_p))$. While that is enough the compute these wrapped Floer cohomologies as vector spaces, it unfortunately loses most of the information in the product structure. In order to compute the embedding: 
\begin{equation*}
    HW(L_j,L_p)\rightarrow HW(\psi(L_j),\psi(L_p)),
\end{equation*}
 we will need to appeal to an extra grading datum that comes from topological aspects of Fukaya-Seidel categories.
\subsection{HMS for projective space, review} \label{HMS-for-cpn}
In this section, all vector spaces are defined over a fixed base field $k$. We review some of the literature pertaining to homological mirror symmetry for projective space $\mathbb{P}^n$. It was studied by Paul Seidel (when $n=2$ in \cite{vanishing-cycles-and-mutations-2}), Abouzaid in \cite{Abouzaid-toric-HMS}, and more recently by Futaki-Ueda in \cite{Futaki-Ueda}. The folklore result discussed in all these references is an equivalence of triangulated categories:
\begin{equation}\label{theta}
    \theta: D^{\pi}\FS((\mathbb{C}^*)^n,W_{\cl}) \rightarrow D^b\Coh(\mathbb{P}^n).
\end{equation}

Because $\mathbb{P}^n$ is Fano, the equivalence above can be fixed (for example) by setting $\theta(\hat{L}_{0,0})=\matheu{O}_{\mathbb{P}^n}$, and then choosing homogeneous coordinates on $\mathbb{P}^n$. Note that $\hat{L}_{0,0}$ is a cotangent fiber of $(\mathbb{C}^*)^n$. This uniqueness of choice in $\theta$ sets some expectations on how the functor $\theta$ should behave, and this section is devoted to establishing some of them. In particular, we provide a more or less topological description of $D^{\pi}FS((\mathbb{C}^*)^n,W_{\cl})$.

\subsubsection{Algebraic computations} In \cite{Futaki-Ueda}, Futaki and Ueda consider a collection of graded Lagrangian thimbles $C_0,C_1,\dots, C_n$ in $FS((\mathbb{C}^*)^n,W_{\cl})$ that we can best describe with the following figure:

\begin{center}

\begin{tikzpicture}[x=0.75pt,y=0.75pt,yscale=-.7,xscale=.7]

\draw   (100,193) .. controls (100,109.05) and (206.56,41) .. (338,41) .. controls (469.44,41) and (576,109.05) .. (576,193) .. controls (576,276.95) and (469.44,345) .. (338,345) .. controls (206.56,345) and (100,276.95) .. (100,193) -- cycle ;
\draw  [fill={rgb, 255:red, 144; green, 19; blue, 254 }  ,fill opacity=1 ] (302,133) .. controls (302,130.24) and (304.24,128) .. (307,128) .. controls (309.76,128) and (312,130.24) .. (312,133) .. controls (312,135.76) and (309.76,138) .. (307,138) .. controls (304.24,138) and (302,135.76) .. (302,133) -- cycle ;
\draw  [fill={rgb, 255:red, 144; green, 19; blue, 254 }  ,fill opacity=1 ] (369,136) .. controls (369,133.24) and (371.24,131) .. (374,131) .. controls (376.76,131) and (379,133.24) .. (379,136) .. controls (379,138.76) and (376.76,141) .. (374,141) .. controls (371.24,141) and (369,138.76) .. (369,136) -- cycle ;
\draw  [fill={rgb, 255:red, 144; green, 19; blue, 254 }  ,fill opacity=1 ] (401,188) .. controls (401,185.24) and (403.24,183) .. (406,183) .. controls (408.76,183) and (411,185.24) .. (411,188) .. controls (411,190.76) and (408.76,193) .. (406,193) .. controls (403.24,193) and (401,190.76) .. (401,188) -- cycle ;
\draw  [fill={rgb, 255:red, 144; green, 19; blue, 254 }  ,fill opacity=1 ] (369,250) .. controls (369,247.24) and (371.24,245) .. (374,245) .. controls (376.76,245) and (379,247.24) .. (379,250) .. controls (379,252.76) and (376.76,255) .. (374,255) .. controls (371.24,255) and (369,252.76) .. (369,250) -- cycle ;
\draw  [fill={rgb, 255:red, 144; green, 19; blue, 254 }  ,fill opacity=1 ] (301,252) .. controls (301,249.24) and (303.24,247) .. (306,247) .. controls (308.76,247) and (311,249.24) .. (311,252) .. controls (311,254.76) and (308.76,257) .. (306,257) .. controls (303.24,257) and (301,254.76) .. (301,252) -- cycle ;
\draw  [fill={rgb, 255:red, 144; green, 19; blue, 254 }  ,fill opacity=1 ] (265,188) .. controls (265,185.24) and (267.24,183) .. (270,183) .. controls (272.76,183) and (275,185.24) .. (275,188) .. controls (275,190.76) and (272.76,193) .. (270,193) .. controls (267.24,193) and (265,190.76) .. (265,188) -- cycle ;
\draw  [fill={rgb, 255:red, 208; green, 2; blue, 27 }  ,fill opacity=1 ] (505,88) .. controls (505,85.24) and (507.24,83) .. (510,83) .. controls (512.76,83) and (515,85.24) .. (515,88) .. controls (515,90.76) and (512.76,93) .. (510,93) .. controls (507.24,93) and (505,90.76) .. (505,88) -- cycle ;
\draw    (406,188) -- (575,188) ;
\draw    (374,250) .. controls (355,109) and (535,205) .. (575,175) ;
\draw    (306,252) .. controls (415,93) and (531,194) .. (571,164) ;
\draw    (270,188) .. controls (572,110) and (479,174) .. (564,146) ;
\draw    (307,133) .. controls (231,192) and (518,127) .. (556,132) ;
\draw    (374,136) .. controls (414,106) and (494,138) .. (534,108) ;

\draw (442,195) node [anchor=north west][inner sep=0.75pt]   [align=left] {$C_0$};
\draw (373,105) node [anchor=north west][inner sep=0.75pt]   [align=left] {$C_5$};
\draw (265,132) node [anchor=north west][inner sep=0.75pt]   [align=left] {$C_4$};
\draw (277,191) node [anchor=north west][inner sep=0.75pt]   [align=left] {$C_3$};
\draw (318,237) node [anchor=north west][inner sep=0.75pt]   [align=left] {$C_2$};
\draw (379,226) node [anchor=north west][inner sep=0.75pt]   [align=left] {$C_1$};
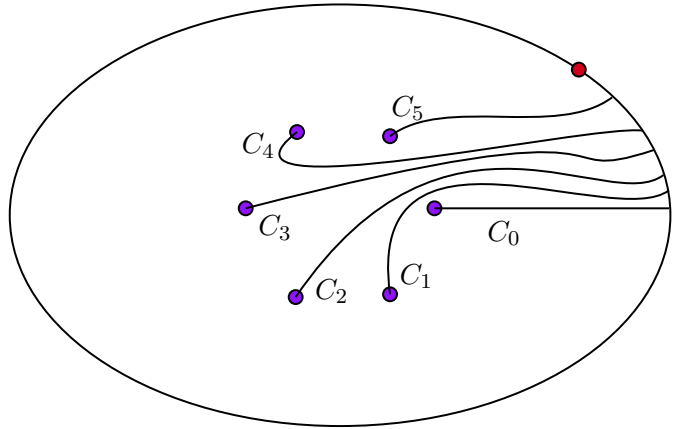
\draw (203,361) node [anchor=north west][inner sep=0.75pt]   [align=left] {Futaki-Ueda thimbles for $n=5$.};
\end{tikzpicture}
\end{center}

Their main theorem is the following computation:
\begin{theorem}(see \cite{Futaki-Ueda})\label{Futaki-Ueda-isom}
Let $V$ be a vector space in degree $0$ of dimension $n+1$. Then for each pair of Lefschetz thimbles $C_i$ and $C_j$, we have an isomorphism of graded vector spaces:
\begin{equation}\label{Futaki-Ueda theorem}
    HW(C_i,C_j) \simeq \bigwedge^{j-i}\left(V[-1]\right).
\end{equation}
Furthermore, these isomorphisms match the triangle product in the Fukaya-Seidel category with the wedge product. The higher $A_{\infty}$-operations all vanish.
\end{theorem}
On the B-side of things, this collections mirrors (a twist of) Beilinson's dual collection, which classically is the full exceptional collection:
\begin{equation*}
    \matheu{C}(-1) = \langle \Omega^n_{\mathbb{P}^n}(n)[n],\Omega^{n-1}_{\mathbb{P}^n}(n-1)[n-1],\dots,\Omega^1_{\mathbb{P}^n}(1)[1],\matheu{O}_{\mathbb{P}^n}\rangle.
\end{equation*}
Because of choices we made on the $A$-side, we twist this collection by $\matheu{O}_{\mathbb{P}^n}(1)$, the resulting collection will then be denoted $\matheu{C}$:
\begin{equation*}
    \matheu{C} =\langle \Omega^n_{\mathbb{P}^n}(n+1)[n],\Omega^{n-1}_{\mathbb{P}^n}(n)[n-1],\dots,\Omega^1_{\mathbb{P}^n}(2)[1],\matheu{O}_{\mathbb{P}^n}(1)\rangle.
\end{equation*}
The $A_{\infty}$-equivalence between the full exceptional collections $\matheu{C}$ and the Lefschetz thimbles
$\langle C_0,\dots,C_n \rangle$, induces an equivalence of triangulated categories as in (\ref{theta}).

The relationship between the collection $\matheu{C}$ and the collection of thimbles $\hat{L}_{k,0}$ we introduced earlier, is Koszul duality.
\begin{lemma}(see \cite{PL-theory}, sections 18k,18l)
In the $A_{\infty}$-category $\FS((\mathbb{C}^*)^n,W_{\cl})$, the collection $\langle \hat{L}_{n,0},\dots ,\hat{L}_{0,0} \rangle$ is the Koszul dual collection to $\langle C_0,\dots,C_n \rangle$. 
\end{lemma}
Koszul duality is customarily denoted with an upper shriek, for example:
\begin{equation*}
    \hat{L}_{k,0} = C^{!}_{k}.
\end{equation*}
As a consequence, the equivalence $\theta$ from (\ref{theta}) above maps $\hat{L}_{k,0}$ to $\matheu{O}_{\mathbb{P}^n}(-k)$, for each $k=0,\dots,n$. This allows us in particular to compute the hom spaces between them:
\begin{equation}\label{hom-spaces}
    HW(\hat{L}_{i,0},\hat{L}_{j,0}) \simeq \text{Sym}^{j-i}(V^{\vee}),
\end{equation}
whenever $i\leq j$. In order the reach other Lefschetz thimbles of the form $\hat{L}_{k,d}$, the tool we need is Serre duality. On the $B$-side, the triangulated category $D^b\Coh(\mathbb{P}^n)$ has a Serre functor given by:
\begin{equation*}
    S(\matheu{L}) = \matheu{L}(-(n+1))[n].
\end{equation*}
On the $A$-side, the Serre functor takes a thimble $L$ to its image under (counter-clockwise) monodromy near infinity, and then shifts the underlying grading by $n$. Another way to think of this monodromy near infinity is wrapping past the stop. A classical result (see for instance Lemma 1.30 in \cite{Fourier-Mukai}) ensures that any triangulated equivalence has to commute with Serre functors. Therefore, it follows that the functor $\theta$ from (\ref{theta}) satisfies:
\begin{equation*}
    \theta(\hat{L}_{k,i}) = \matheu{O}_{\mathbb{P}^n}(-k+i(n+1)).
\end{equation*}
To simplify notation a bit, we now will denote by $\hat{L}_d$ (for $d\in\mathbb{Z}$) any Lagrangian thimble whose image under $\theta$ is $\matheu{O}_{\mathbb{P}^n}(d)$. By means of Serre duality, we can now compute the hom space between all thimbles $\hat{L}_d$. For example:
\begin{equation*}
    \hom(\hat{L}_0,\hat{L}_{-d}) \simeq \text{Sym}^{d-(n+1)}(V)[n],
\end{equation*}
whenever $d\geq n+1$. We also note that these isomorphisms respect the product structures too.

\subsubsection{Topological computations} We begin with the observation that the $A_{\infty}$-category $\FS((\mathbb{C}^*)^n,W_\cl)$ carries a topological grading by the relative homology group:
\begin{equation}\label{fine-grading}
    \hat{G} = H_1((\mathbb{C}^*)^n, \text{Crit}(W_\cl),\mathbb{Z}),
\end{equation}
where $\text{Crit}(W_\cl)$ is the (finite) collection of critical points of $W_\cl$. This grading associates with each Hamiltonian $y:[0,1]\rightarrow (\mathbb{C}^*)^n$ from a Lefschetz thimble $L$ to another Lefschetz thimble $L'$, an element $\deg_{\hat{G}}(y)\in \hat{G}$ by connecting $y(0)$ to the vanishing point of $L$ (without leaving $L$), and $y(1)$ to the vanishing point of $L'$ (without leaving $L'$) and then taking the homology class of the resulting path in $\hat{G}$. Because of its topological nature, this $\hat{G}$-grading is preserved by all Floer theoretic constructions. This includes continuation maps, TQFT structures, $A_{\infty}$-operations, twists and mutations.

This topological grading however, is a bit too fine for our purposes: For example, in the computation of Futaki-Ueda \ref{Futaki-Ueda-isom}, the vector space $V$ inherits different $\hat{G}$-gradings from the different isomorphisms:
\begin{equation*}
    HW(C_k,C_{k+1}) \simeq V[-1].
\end{equation*}
We can remedy this issue by identifying all $n+1$ critical points of $W_\cl$ in the homology group defining $\hat{G}$ (see (\ref{fine-grading})). We do so by means of the projection map:
\begin{equation*}
    \pi: \hat{G} \rightarrow H_1((\mathbb{C}^*)^n,x_0), 
\end{equation*}
where $x_0$ is the unique non-degenerate critical point of $W$. Observe that the group:
\begin{equation*}
    G = H_1((\mathbb{C}^*)^n,x_0),
\end{equation*}
naturally grades the Fukaya-Seidel category $\FS((\mathbb{C}^*)^n,W)$, and the collapsing map $\pi:\hat{G}\rightarrow G$ makes $\FS((\mathbb{C}^*)^n,W_\cl)$ a $G$-graded $A_{\infty}$-category as well.
\begin{remark}
The group $G$ is isomorphic to $\mathbb{Z}^n$ but we are not fixing an isomorphism yet. The $G$-grading on the $A$-side should be compared with the toric grading on $D^b\Coh(\mathbb{P}^n)$ in the $B$-side (see \cite{Derived-McKay} for instance).
\end{remark}
\begin{lemma}\label{graded-hom-spaces}
There is a $G$-grading on the vector space $V$ so that the isomorphisms in (\ref{Futaki-Ueda theorem}) are all $G$-graded.
\end{lemma}
\begin{proof}
This is best seen from the isomorphisms in (\ref{hom-spaces}), because the $G$-grading in $HW(\hat{L}_{k+1,0},\hat{L}_{k,0})$ is inherited from the one in $HW(L_0,L_1)$ via the map $\pi$, independently of $k=0,1,\dots,n-1$. It follows that $V$ has a $G$-grading such that the isomorphisms:
\begin{equation*}
 HW(\hat{L}_{k+1,0},\hat{L}_{k,0})\simeq V,
\end{equation*}
are $G$-graded for $k=0,\dots,n$. Using the Serre functor, we can take any integer $d\in\mathbb{Z}$, and isotope the pair $(L_d,L_{d+1})$ past the stop sufficiently many times in order to get get an isomorphism:
\begin{equation*}
    HW(\hat{L}_d,\hat{L}_{d+1})\simeq HW(\hat{L}_{k+1,0},\hat{L}_{k,0}),
\end{equation*}
for some $k=0,\dots,n$. As a consequence, the isomorphism:
\begin{equation*}
    HW(\hat{L}_d,\hat{L}_{d+1})\simeq V,
\end{equation*}
is $G$-graded for all $d\in\mathbb{Z}$. Next, whenever $i<j$, we have a $G$-graded surjective map:
\begin{equation*}
    HW(\hat{L}_{j-1},\hat{L}_j)\otimes \dots \otimes HW(\hat{L}_i,\hat{L}_{i+1}) \rightarrow HW(\hat{L}_i,\hat{L}_j),
\end{equation*}
given by iterated composition (not to be confused with the $A_{\infty}$-structure maps). Because this map is surjective, one deduces that the isomorphisms in (\ref{hom-spaces}) all respect the $G$-grading. Now the lemma follows from an application of Koszul duality to the collection $(\hat{L}_{n,0},\dots,\hat{L}_{0,0})$.
\end{proof}

We now consider the \emph{weight} decomposition of $V$ with respect to $G$:
\begin{equation}\label{weight-decomposition}
    V = \ell_{g_0}\oplus \ell_{g_1}\oplus\dotsi\oplus \ell_{g_n},
\end{equation}
where $g_0,\dots,g_n$ are elements of $G$, and $\ell_{g}$ denotes a one dimensional vector space where all non-zero elements have degree $g$. We will see later that in this decomposition, all $g_k$ are distinct, but we do not assume that for now.

\begin{lemma}
In the group $G$, we have the following relation:
\begin{equation*}
    g_0+g_1+\dots+g_n = 0.
\end{equation*}
\end{lemma}
\begin{proof}
This Lemma is purely topological, but we exploit known Floer theoretic calculations to prove it. From the Koszul duality isomorphism in Lemma 5.15, (ii) of \cite{PL-theory}, we have a $G$-graded isomorphism:
\begin{equation*}
    \hom(C_0,\hat{L}^!_{n})\simeq \hom(\hat{L}_0,\hat{L}_1[n])^{\vee},
\end{equation*}
because $\hat{L}^!_n = C_n$. We therefore get a $G$-graded isomorphism:
\begin{equation*}
    \wedge^n V \simeq V^{\vee}.
\end{equation*}
Now the lemma follows by comparing the sum of the weights (as in (\ref{weight-decomposition})) appearing on both sides of the isomorphism above.
\end{proof}
\begin{lemma}
The group $G$ has the following presentation:
\begin{equation*}
    G = \mathbb{Z}g_0\oplus\dotsi\oplus\mathbb{Z}g_n /\langle g_0+\dotsi+g_n \rangle .
\end{equation*}
\end{lemma}
\begin{proof}
Because of the previous lemma, together with the fact that $G$ is a free abelian group of rank $n$, it suffices to show that the elements $g_i$ generate the group $G$. Let $G'\subseteq G$ be the subgroup generated by $g_0,\dots,g_n$. Because of Lemma \ref{graded-hom-spaces}, all of the partially wrapped Floer cohomology vector spaces $HW(\hat{L}_i,\hat{L}_j)$ are $G'$-graded. Next, using the isomorphisms:
\begin{equation*}
    HW(L_i,L_j) \rightarrow HW(\hat{L}_{0,i},\psi(\hat{L}_j)),
\end{equation*}
we deduce that the cohomology vector spaces $HW(L_i,L_j)$ are also $G'$-graded. At the same time, the wrapping sequence:
\begin{equation*}
L_0\rightarrow L_1 \rightarrow \dotsi \rightarrow L_i\rightarrow \dotsi   
\end{equation*}
computes the fully (unstopped) wrapped Floer cohomology algebra $\matheu{W}(L_0)$ of $L_0$, as the limit:
\begin{equation*}
    \varinjlim_i HW(L_0,L_i) = \matheu{W}(L_0).
\end{equation*}
It follows that the (unstopped) wrapped Floer cohomology is also $G'$-graded. However, the later is canonically given by:
\begin{equation*}
\matheu{W}(L_0) \simeq k[G],    
\end{equation*}
where the right hand side is the group algebra of $G$. As a consequence, $G'=G$ and the Lemma follows.
\end{proof}
we reorganize all of the previous discussion in the following theorem.
\begin{theorem}
There is a group isomorphism $\alpha:G\rightarrow \mathbb{Z}^n$ and an equivalence of triangulated categories:
\begin{equation*}
    \theta: D^{\pi}\FS((\mathbb{C}^*)^n,W_{\cl}) \rightarrow D^b\Coh(\mathbb{P}^n),
\end{equation*}
with the following properties:
\begin{itemize}
    \item[-] At the level of objects, we have $\theta(\hat{L}_{k,i}) = \matheu{O}_{\mathbb{P}^n}(-k+i(n+1))$. We also use the notation $\hat{L}_d=\hat{L}_{k,i}$ whenever $d=-k+i(n+1)$.
    \item[-] At the level of hom-spaces, the linear isomorphisms: 
    $$\theta: \hom(\hat{L}_i,\hat{L}_j)\rightarrow \hom(\matheu{O}_{\mathbb{P}^n}(i),\matheu{O}_{\mathbb{P}^n}(j))$$
    map a Hamiltonian chord of topological degree $g\in G$, to the monomial $x^{\alpha(g)}$.
\end{itemize}
\end{theorem}
\begin{remark}
In item 2 of the previous theorem, in the case where $j<i$, we still think of $\hom(\matheu{O}_{\mathbb{P}^n}(i),\matheu{O}_{\mathbb{P}^n}(j))$ as a vector space of monomials by means of Serre duality:
\begin{equation*}
    \hom(\matheu{O}_{\mathbb{P}^n}(i),\matheu{O}_{\mathbb{P}^n}(j)) \simeq \hom(\matheu{O}_{\mathbb{P}^n}(j),\matheu{O}_{\mathbb{P}^n}(i-n-1))^{\vee}[n].
\end{equation*}
\end{remark}
\subsection{B-side calculations}
We now carry out some calculations on the algebraic geometry side of homological mirror symmetry to understand the category $\Perf(X_0)$. We will heavily rely on the structure of the cyclic covering map $\phi:X_0\rightarrow \mathbb{P}^n$ and the action of $\mathbb{Z}_{n+1}$ on $X_0$ as deck transformations. To begin with, observe that for any coherent sheaf $\matheu{G}$ on $X_0$, we have a natural isomorphism of sheaf cohomology:
\begin{equation}\label{adjunction-iso}
    \hom^i(\matheu{O}_{X_0},\matheu{G})\rightarrow \hom^i(\matheu{O}_{\mathbb{P}^n},\phi_*\matheu{G}).
\end{equation}
It comes from a composition of the pushforward map:
\begin{equation*}
 \hom^i(\matheu{O}_{X_0},\matheu{G})\rightarrow \hom^i(\phi_* \matheu{O}_{X_0},\phi_*\matheu{G}),   
\end{equation*}
with the structure map $\iota:\matheu{O}_{\mathbb{P}^n}\rightarrow \phi_* \matheu{O}_{X_0}$. Because $\phi$ is a cyclic covering, we actually have an isomorphism of $\matheu{O}_{\mathbb{P}^n}$-modules:
\begin{equation}\label{weight decomposition}
    \phi_*\matheu{O}_{X_0} \simeq \mathscr{E},
\end{equation}
where $\mathscr{E}$ is the locally free sheaf:
\begin{equation*}
   \mathscr{E} = \matheu{O}_{\mathbb{P}^n}\oplus \matheu{O}_{\mathbb{P}^n}(-1)\oplus \dotsi\oplus \matheu{O}_{\mathbb{P}^n}(-n).
\end{equation*}
This isomorphism endows $\mathscr{E}$ with the structure of a sheaf of $\matheu{O}_{\mathbb{P}^n}$-algebras, which in turn completely determines $X_0$. We also remind the reader that the vector bundle $\mathscr{E}$  split-generates the triangulated category $D^b\Coh(\mathbb{P}^n)$. We fix an injective resolution $I$ of the structure sheaf $\matheu{O}_{X_0}$, and we use it to build a dg-model $\mathscr{C}_{\dg}$ for $\Perf(X_0)$ as follows:
\begin{equation}\label{dg-model}
    \mathscr{C}_{\dg}(i,j) = \hom_{X_0}^{\bullet}(I(i),I(j)).
\end{equation}
Because $\phi$ is a finite map, the sheaf $\phi_*I$ is an injective resolution for $\mathscr{E}$. We can therefore use it to produce a dg-model for  $\mathbb{P}^n$ as well:
\begin{equation*}
 \mathscr{A}_{\dg}(i,j) = \hom_{\mathbb{P}^n}^{\bullet}(\phi_*I(i),\phi_*I(j)). 
\end{equation*}
Note in particular that we have a dg-pushforward map:
\begin{equation*}
    \phi_*: \mathscr{C}_{\dg} \rightarrow \mathscr{A}_{\dg}.
\end{equation*}
At the level of cohomology, this functor becomes a faithful (but not full) embedding $H(\phi):H(\mathscr{C}_{\dg}) \rightarrow H(\mathscr{A}_\dg)$. The next lemma shows an instance of how the image of $H(\phi)$ remembers the cyclic covering it came from.

\begin{lemma}\label{pushforward-calculation}
Let $X_f = V(t^{n+1}-f(x_0,\dots,x_n))\subseteq \mathbb{P}^{n+1}$ be a degree $n+1$ hypersurface, and let $\phi:X_f\rightarrow\mathbb{P}^n$ be the branched covering map that "forgets t". The pushforward of the homomorphism $(-)\times t: \matheu{O}_{X_f}\rightarrow \matheu{O}_{X_f}(1)$ using the covering map $\phi$ has the formula:
\begin{equation*}
    \phi_*((-)\times t) = \emph{\id}_{\matheu{O}}\oplus \emph{\id}_{\matheu{O}(-1)} \oplus \dotsi \oplus \emph{\id}_{\matheu{O}(-n+1)}\oplus (\matheu{O}(-n)\xrightarrow{(-)\times f} \matheu{O}(1)).
\end{equation*}
\end{lemma}
\begin{proof}
Let $R=\mathbb{C}[x_0,\dots,x_n]$ and $S=R[t]/(t^{n+1}-f)$ be the homogeneous coordinate rings defining the varieties $\mathbb{P}^n$ and $X_f$, respectively. Then the line bundle decomposition in (\ref{weight decomposition}) is the sheafy version of the direct sum decomposition of graded $R$-modules:
\begin{equation*}
    S = R \oplus R(-1)\oplus \dotsi\oplus R(-n),
\end{equation*}
where the inclusion $R(-k)\rightarrow S$ is multiplication by $t^k$. The Lemma then follows from interpreting the map $(-)\times t \in \hom_S(S,S(1))$ in terms of this decomposition.
\end{proof}

The previous lemma (at least in principle) is enough to determine the entire image of the functor $H(\phi)$. However, there is another approach that we favor in doing this computation, and it involves extra grading data that our categories come with.

We now explain how the categories $H(\mathscr{A}_\dg)$ and $H(\mathscr{C}_\dg)$ carry a grading by $\mathbb{Z}^{n}$ that we call the \emph{toric grading}. We begin by fixing an action of $T=(\mathbb{C}^*)^n$ on $\mathbb{P}^n$ and $X_0$ as follows:
\begin{align}
    (\zeta_1,\dots,\zeta_n)\cdot [x_0:\dots:x_n] &= [\zeta^{-1} x_0:\zeta_1x_1:\dots:\zeta_nx_n] \ \ \text{on}\ \mathbb{P}^n, \\
    (\zeta_1,\dots,\zeta_n)\cdot [t:x_0:\dots:x_n] &= [t:\zeta^{-1} x_0:\zeta_1x_1:\dots:\zeta_nx_n]\ \ \text{on}\ X_0,\notag
\end{align}
where:
\begin{equation*}
    \zeta = \zeta_1\zeta_2\dotsi\zeta_n.
\end{equation*}
Note in particular that $\phi:X_0\rightarrow\mathbb{P}^n$ is $T$-equivariant.

Let $Y$ be a projective variety with an action of $T$ on it. This action produces a \emph{consistent} choice of isomorphisms for all $\zeta\in T$:
\begin{align*}
    \matheu{O}_Y &\rightarrow \zeta^*\matheu{O}_Y \\
    g&\mapsto \zeta^*g,
\end{align*}
that pulls-back regular functions on open subsets of $Y$ using the torus action. This \emph{consistent} choice of isomorphisms is called a \emph{linearization}; we refer the reader to \cite{Derived-McKay} for a more detailed treatment of this idea. If $D\subseteq Y$ is $T$-invariant divisor, then we can similarly pull-back meromorphic functions to produce a linearization of $\matheu{O}_Y(D)$. When two coherent sheaves $\matheu{F}$ and $\matheu{G}$ are linearized, the vector space $\hom_Y(\matheu{F},\matheu{G})$ carries a $T$-action via the diagram:

\begin{equation}\label{commutative-diagram}
    \begin{tikzcd}
    \matheu{F} \arrow[r,"\sigma"]\arrow[d]&
    \matheu{G}\arrow[d]&\\
   t^*\matheu{F}\arrow[r,"t^*\sigma"]& t^*\matheu{G}. &
    \end{tikzcd}
\end{equation}

As a consequence, the finite dimensional $T$-representation $\hom_Y(\matheu{F},\matheu{G})$ carries a weight-decomposition, which is the toric grading by $\mathbb{Z}^{n}$ that we have alluded to before. By specializing the previous discussion to $Y=\mathbb{P}^n$, and then to $Y=X_0$, we deduce the following:
\begin{lemma}
The categories $H(\mathscr{A}_\dg)$ and $H(\mathscr{C}_{\emph{\dg}})$ carry toric gradings by $\mathbb{Z}^{n}$. Furthermore, because $\phi$ is $T$-equivariant, the functor $H(\phi)$ respects this grading.\qed
\end{lemma}

Going back to the discussion following Lemma \ref{pushforward-calculation}, we get a practical description of the pushforward map as follows:
\begin{lemma}\label{pushforward-calculation-2}
For each integer $d$, and $v\in\mathbb{Z}^n$, there is at most one monomial in $\hom_{X_0}(\matheu{O}_{X_0},\matheu{O}_{X_0}(d))$ whose toric degree is $v$. Moreover, when such a monomial exists, its pushforward using $\phi$ is the sum of all $n+1$ monomials of degree $v$ in the direct sum decomposition of $\hom_{\mathbb{P}^n}(\mathscr{E},\mathscr{E}(d))$.
\end{lemma}
\begin{proof}
Consider two degree $d\geq 0$ monomials on $X_0$:
\begin{equation*}
    t^{\alpha}x_0^{\alpha_0}\dots x_n^{\alpha_n} \ \ \ \text{and}\ \ \ t^{\beta}x_0^{\beta_0}\dots x_n^{\beta_n}.
\end{equation*}
Their toric degrees (respectively) are $(\alpha_1-\alpha_0,\dots,\alpha_n-\alpha_0)$ and $(\beta_1-\beta_0,\dots,\beta_n-\beta_0)$. For the two toric degrees to agree, we need the difference $\alpha_k-\beta_k$ to be independent of $k=0,1,\dots,n$. At the same, the two monomials have the same polynomial degree $d$. It follows that:
\begin{equation*}
    \beta-\alpha = (n+1)(\alpha_k-\beta_k),
\end{equation*}
for all $k=0,1,\dots,n$. We can however show using these identities that:
\begin{equation*}
    \frac{t^{\beta}x_0^{\beta_0}\dots x_n^{\beta_n}}{t^{\alpha}x_0^{\alpha_0}\dots x_n^{\alpha_n}} = \left(\frac{t^{n+1}}{x_0\dots x_n}\right)^{\alpha_0-\beta_0}.
\end{equation*}
It follows that the two monomials are equal in $\hom_{X_0}(\matheu{O}_{X_0},\matheu{O}_{X_0}(d))$. 

The second part of the Lemma can be proved in exactly the same way as Lemma \ref{pushforward-calculation}.
Finally, the case $d<0$ follows from Serre duality which also respects the toric grading:
\begin{equation*}
 \hom_{X_0}(\matheu{O}_{X_0}(d),\matheu{O}_{X_0}(-1))\otimes \hom_{X_0}(\matheu{O}_{X_0},\matheu{O}_{X_0}(d))\rightarrow k[n].    
\end{equation*}
\end{proof}
By identifying the toric grading on perfect complexes with the topological grading on Fukaya-Seidel categories, we prove the following upgrade of the isomorphism in Lemma \ref{A side isomorphism}.

\begin{lemma}
For each pair of integers $i$ and $j$, the two embeddings:
\begin{equation*}
    HW(L_i,L_j)\xrightarrow{\theta\circ\psi} \hom_{\mathbb{P}^n}(\mathscr{E}(i),\mathscr{E}(j)) \xleftarrow{\phi_*} \hom_{X_0}(\matheu{O}_{X_0}(i),\matheu{O}_{X_0}(j)),
\end{equation*}
have the same image.
\end{lemma}

\begin{proof}
Indeed, let $p\in HW(L_i,L_j)$ be an intersection point of topological degree $g\in G$. By definition:
\begin{equation*}
    \psi(p) = p_0+p_1+\dots +p_n,
\end{equation*}
is the sum of all intersection points in $HW(\hat{L}_{k,i_k},\hat{L}_{l,j_l})$ of topological degree $g$. It follows that in the decomposition:
\begin{equation*}
\hom_{\mathbb{P}^n}(\mathscr{E}(i),\mathscr{E}(j)) = \bigoplus_{0\leq k,l\leq n} \hom_{\mathbb{P}^n}(\matheu{O}(i-k),\matheu{O}(j-k)),
\end{equation*}
the element $\theta\circ\psi(p)$ is the sum of all monomials of degree $\alpha(g)\in \mathbb{Z}^n$. But, as in Lemma \ref{pushforward-calculation-2}, this is exactly the image under $\phi_*$ of the unique monomial in $\hom_{X_0}(\matheu{O}_{X_0}(i),\matheu{O}_{X_0}(j))$ whose degree toric degree is $\alpha(g)\in\mathbb{Z}^n$.
\end{proof}
\hspace{-1cm}{
\tikzset{every picture/.style={line width=0.75pt}} 
\begin{tikzpicture}[x=0.75pt,y=0.75pt,yscale=-.9,xscale=.9]

\draw   (100,153) .. controls (100,78.44) and (205.66,18) .. (336,18) .. controls (466.34,18) and (572,78.44) .. (572,153) .. controls (572,227.56) and (466.34,288) .. (336,288) .. controls (205.66,288) and (100,227.56) .. (100,153) -- cycle ;
\draw    (230,32) -- (298,117) ;
\draw    (299,190) -- (239,276) ;
\draw    (392,153.5) -- (572,153) ;
\draw    (392,153.5) .. controls (439,263) and (136,189) .. (239,276) ;
\draw    (230,32) .. controls (387,42) and (45,198) .. (299,190) ;
\draw    (572,153) .. controls (421,220) and (482,-28) .. (298,117) ;
\draw  [fill={rgb, 255:red, 189; green, 16; blue, 224 }  ,fill opacity=1 ] (292.5,117) .. controls (292.5,113.96) and (294.96,111.5) .. (298,111.5) .. controls (301.04,111.5) and (303.5,113.96) .. (303.5,117) .. controls (303.5,120.04) and (301.04,122.5) .. (298,122.5) .. controls (294.96,122.5) and (292.5,120.04) .. (292.5,117) -- cycle ;
\draw  [fill={rgb, 255:red, 189; green, 16; blue, 224 }  ,fill opacity=1 ] (386.75,153.5) .. controls (386.75,150.6) and (389.1,148.25) .. (392,148.25) .. controls (394.9,148.25) and (397.25,150.6) .. (397.25,153.5) .. controls (397.25,156.4) and (394.9,158.75) .. (392,158.75) .. controls (389.1,158.75) and (386.75,156.4) .. (386.75,153.5) -- cycle ;
\draw  [fill={rgb, 255:red, 189; green, 16; blue, 224 }  ,fill opacity=1 ] (293.5,190) .. controls (293.5,186.96) and (295.96,184.5) .. (299,184.5) .. controls (302.04,184.5) and (304.5,186.96) .. (304.5,190) .. controls (304.5,193.04) and (302.04,195.5) .. (299,195.5) .. controls (295.96,195.5) and (293.5,193.04) .. (293.5,190) -- cycle ;
\draw    (186,213) -- (274,213) ;
\draw [shift={(276,213)}, rotate = 180] [color={rgb, 255:red, 0; green, 0; blue, 0 }  ][line width=0.75]    (10.93,-3.29) .. controls (6.95,-1.4) and (3.31,-0.3) .. (0,0) .. controls (3.31,0.3) and (6.95,1.4) .. (10.93,3.29)   ;
\draw    (186,204) -- (235.04,193.42) ;
\draw [shift={(237,193)}, rotate = 527.8299999999999] [color={rgb, 255:red, 0; green, 0; blue, 0 }  ][line width=0.75]    (10.93,-3.29) .. controls (6.95,-1.4) and (3.31,-0.3) .. (0,0) .. controls (3.31,0.3) and (6.95,1.4) .. (10.93,3.29)   ;
\draw    (470,215) -- (435.11,162.66) ;
\draw [shift={(434,161)}, rotate = 416.31] [color={rgb, 255:red, 0; green, 0; blue, 0 }  ][line width=0.75]    (10.93,-3.29) .. controls (6.95,-1.4) and (3.31,-0.3) .. (0,0) .. controls (3.31,0.3) and (6.95,1.4) .. (10.93,3.29)   ;
\draw    (457,228) -- (385.94,210.48) ;
\draw [shift={(384,210)}, rotate = 373.85] [color={rgb, 255:red, 0; green, 0; blue, 0 }  ][line width=0.75]    (10.93,-3.29) .. controls (6.95,-1.4) and (3.31,-0.3) .. (0,0) .. controls (3.31,0.3) and (6.95,1.4) .. (10.93,3.29)   ;
\draw  [fill={rgb, 255:red, 208; green, 2; blue, 27 }  ,fill opacity=1 ] (509.75,64.5) .. controls (509.75,61.6) and (512.1,59.25) .. (515,59.25) .. controls (517.9,59.25) and (520.25,61.6) .. (520.25,64.5) .. controls (520.25,67.4) and (517.9,69.75) .. (515,69.75) .. controls (512.1,69.75) and (509.75,67.4) .. (509.75,64.5) -- cycle ;

\draw (125,203) node [anchor=north west][inner sep=0.75pt]   [align=left] {$\matheu{O}_{\mathbb{P}^n}(-1)$};
\draw (465,220) node [anchor=north west][inner sep=0.75pt]   [align=left] {$\matheu{O}_{\mathbb{P}^n}$};
\draw (184,53) node [anchor=north west][inner sep=0.75pt]   [align=left] {$\matheu{O}_{\mathbb{P}^n}(-2)$};
\draw (419,55) node [anchor=north west][inner sep=0.75pt]   [align=left] {$\matheu{O}_{\mathbb{P}^n}(1)$};
\draw (388,130) node [anchor=north west][inner sep=0.75pt]   [align=left] {$\id$};
\draw (309,180) node [anchor=north west][inner sep=0.75pt]   [align=left] {$\id$};
\draw (268,125) node [anchor=north west][inner sep=0.75pt]   [align=left] {$(-)\cdot x_0x_1x_2$};
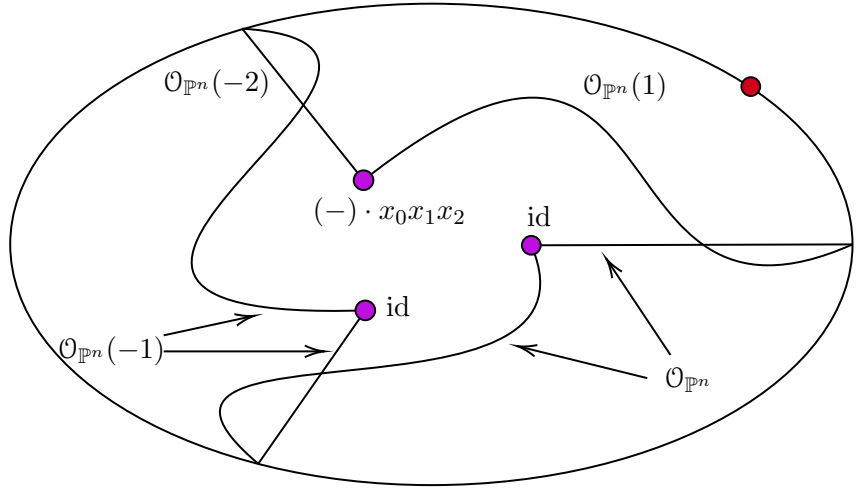
\draw (194,302) node [anchor=north west][inner sep=0.75pt]   [align=left] {Computation of $\phi_*(t)$ on the A-side \\when $n=2$; compare with Lemma \ref{pushforward-calculation}.};
\end{tikzpicture}}
\\

\textit{Proof of Theorem \ref{hypersurface is B}.} In our setup, we have the following diagram of $A_{\infty}$-functors:
\begin{equation*}
    FS((\mathbb{C}^*)^n,W)\xrightarrow{\theta\circ\psi} \mathscr{A}_\dg \xleftarrow{\phi_*} \mathscr{C}_\dg.
\end{equation*}
Recall that the differential graded categories $\mathscr{A}_\dg$ and $\mathscr{C}_\dg$ compute $D^b\Coh(\mathbb{P}^n)$ and $\Perf(X_0)$, respectively. Using the result of the previous lemma, the functors $H(\theta\circ\psi)$ and $H(\phi_*)$ have identical images inside $H(\mathscr{A}_\dg)$, and the desired theorem follows as a consequence. \qed

\bibliographystyle{plain}
\bibliography{bibi.bib}
\end{document}